\documentclass[version=preprint]{iacrcc} 
\license{CC-by}

\newif\iffullversion
\fullversiontrue
\allowdisplaybreaks
\DeclareMathOperator{\ord}{ord}
\newcommand{\G}{\mathbb{G}}
\newcommand{\Z}{\mathbb{Z}}

\newcommand{\F}{\mathbb{F}}
\newcommand{\EC}{\mathcal{E}}
\newcommand{\D}{\mathcal{D}}

\usepackage[caption=false]{subfig} 
\usepackage{multirow}
\usepackage{booktabs}
\usepackage{xy}

\usepackage{enumitem}
\begin{document}

\title[running = {Cycles of supersingular elliptic curves for pairing-based proof systems},
]{Cycles of supersingular elliptic curves for pairing-based proof systems}

\addauthor[inst=1,
           orcid={0000-0001-5423-7714},
           email={craig.costello@qut.edu.au},
           ]{Craig Costello}
\addauthor[inst=2,
           orcid={0000-0003-2452-6165},
           email={gkorpal@arizona.edu}
           ]{Gaurish Korpal}

\addaffiliation[country={Australia}
                ]{Queensland University of Technology}
\addaffiliation[country={USA}
                ]{University of Arizona}

\maketitle


\begin{abstract}
We give new constructions of cycles of pairing-friendly elliptic curves with a view towards unbounded recursive pairing-based proof systems. Unlike the only known prior cycle of elliptic curves - the \emph{ordinary} MNT cycle - our construction uses elliptic curves that are \emph{supersingular}. 

A trade-off of our approach is that the supersingular cycles are defined over extension fields (quadratic extensions in the optimal case), which makes elements and computations in $\G_1$ less compact and efficient than those in the MNT cycle. On the other hand, the supersingular cycles in this paper offer a key advantage over their MNT counterpart: every instance of our infinite family of supersingular curves can be efficiently constructed via Br{\"o}ker's algorithm, whereas it is only feasible to construct a relatively small, bounded number of MNT instances via the CM method. In other words, while both constructions give infinite families of pairing-friendly curves in theory, only the supersingular construction gives rise to infinite numbers of cycles that can be realised in practice. 

Supersingular cycles offer benefits that are relevant in the context of recursive pairing-based proof systems. They afford flexibility in the choices of underlying finite fields; one can choose primes $p$ for which the underlying field arithmetic is efficient and for which $p-1$ is divisible by a large power of 2. Or, as we study in detail, using supersingular cycles unlocks the possibility of connecting the cycle with other pairing-friendly elliptic curves that are defined over much smaller finite fields, where proof system arithmetic is much more efficient. Indeed, constructing these so-called \emph{lollipops} of pairing-friendly curves was the motivating problem (posed by researchers back in 2019) that inspired the present work. 
\end{abstract}

 \begin{textabstract}
 We give new constructions of cycles of pairing-friendly elliptic curves with a view towards unbounded recursive pairing-based proof systems. Unlike the only known prior cycle of elliptic curves - the \emph{ordinary} MNT cycle - our construction uses elliptic curves that are \emph{supersingular}. 

 The drawback of our approach is that the supersingular cycles are defined over extension fields (quadratic extensions in the optimal case), which makes elements and computations in $\mathbb{G}_1$ less compact and efficient than those in the MNT cycle. On the other hand, the supersingular cycles in this paper offer a key advantage over their MNT counterpart: every instance of our infinite family of supersingular curves can be efficiently constructed via Br{\"o}ker's algorithm, whereas it is only feasible to construct a relatively small, bounded number of MNT instances via the CM method. In other words, while both constructions give infinite families of pairing-friendly curves in theory, only the supersingular construction gives rise to infinite numbers of cycles that can be realised in practice. 

 Supersingular cycles offer benefits that are relevant in the context of recursive pairing-based proof systems. They afford flexibility in the choices of underlying finite fields; one can choose primes $p$ for which the underlying field arithmetic is efficient and for which $p-1$ is divisible by a large power of 2. Or, as we study in detail, using supersingular cycles unlocks the possibility of connecting the cycle with other pairing-friendly elliptic curves that are defined over much smaller finite fields, where proof system arithmetic is much more efficient. Indeed, constructing these so-called \emph{lollipops} of pairing-friendly curves was the motivating problem (posed by researchers back in 2019) that inspired the present work. 

 \end{textabstract}

\section{Introduction}\label{sec:intro}

Numerous constructions of \emph{2-cycles} of elliptic curves are now deployed as the foundation of succinct non-interactive arguments of knowledge (SNARKs) -- see~\cite{AHG22} for an extensive survey. Such $2$-cycles involve two elliptic curves, $E/\F_p$ and $\hat{E}/\F_q$, with $p \approx q$ such that $p = \#\hat{E}(\F_q)$ and $q=\#E(\F_p)$, and fall into one of three categories:

\begin{enumerate}[label=(\roman*)]
\item {\it Both $E$ and $\hat{E}$ are pairing-friendly.}\label{pairingcycle} These cycles were first proposed for use in scalable pairing-based SNARKs by Ben{-}Sasson, Chiesa, Tromer and Virza~\cite{scalable} and use instances of the only known cycle of ordinary pairing-friendly curves coming from the Miyaji-Nakabayashi-Takano (MNT) construction~\cite{MNT,DBLP:conf/ants/KarabinaT08}. For example, the Mina protocol~\cite{mina} was first built on top of a cycle of MNT curves. 
\vspace{0.3cm}
\item {\it One of $E$ and $\hat{E}$ is pairing-friendly.}\label{hybridcycle}  These \emph{hybrid} cycles can be readily constructed by taking any pairing-friendly curve $E/\F_p$ of prime order $q$, e.g. a Barreto-Naehrig (BN) curve~\cite{BNcurveViray}, and partnering it with the non-pairing-friendly\footnote{The only known exception here is when $E$ is a particular instance of an MNT curve, in which case $\hat{E}$ can also be pairing-friendly and the cycle would then fall into category~\ref{pairingcycle}.} curve $\hat{E}/\F_q$ of prime order $p$, which is not only guaranteed to exist, but necessarily has the same CM discriminant as $E$ (c.f.~\cite{silvermanstange}). Examples of hybrid cycles found in the wild are Hopwood's Pluto/Eris cycle~\cite{pluto-eris}, Meckler's BN382 cycle~\cite{BN382}, and Williamson's BN254/Grumpkin cycle~\cite{goblinplonk}.
\vspace{0.3cm}
\item {\it Neither $E$ nor $\hat{E}$ are pairing-friendly.}\label{nonpairingcycle} Proof systems that avoid the use of pairings altogether can still exploit the recursive composition afforded by a cycle; the Bulletproofs~\cite{bulletproofs} system is one such popular example. Non-pairing-friendly cycles found in the wild include Poelstra's secp/secq cycle~\cite{secpsecq}, Bowe, Grigg and Hopwood's Tweedledee/Tweedledum cycle~\cite{BoweGH19}, and Hopwood's Pasta cycle~\cite{pasta}.
\end{enumerate} 

The $2$-cycles in cases~\ref{hybridcycle} and~\ref{nonpairingcycle} above (where at least one of the curves is not pairing-friendly) are much easier to find and construct than the $2$-cycles in case~\ref{pairingcycle}. The instantiations falling into cases~\ref{hybridcycle} and~\ref{nonpairingcycle} that are cited above involve curves whose underlying field sizes are either optimally small, or else very close to it. On the other hand, the extra restrictions imposed by insisting that both $E$ and $\hat{E}$ are pairing-friendly make constructing the cycles in case~\ref{pairingcycle} notoriously difficult; beyond the MNT construction, there remain no known methods of constructing cycles of ordinary pairing-friendly elliptic curves~\cite{chiesa2019cycles,revisitingcycles,primepairs}.

The reason many SNARK instantiations opt for cycles where one or both of $E$ and $\hat{E}$ are pairing-friendly is that pairing-based proof systems offer a number of advantages over non-pairing-based proof systems. Groth's SNARK~\cite{Groth16}, often dubbed {\sf Groth16}, is perhaps the most ubiquitous pairing-based proof system in both the academic literature and in practical SNARK implementations. The reason is that the proof sizes in Groth's construction are \emph{constant}, i.e. are independent of the size of the witness/statement they are proving. On the other hand, pairing-free proof systems like Bulletproofs~\cite{bulletproofs} are currently (at best) logarithmic in the size of the witness. Depending on the target application, the resulting proof sizes in these protocols might be acceptable, but in terms of the succinctness property that is fundamental to the real-world appeal of SNARKs, {\sf Groth16} and its pairing-based variants remain unrivalled by their non-pairing-based counterparts. 

\paragraph{Drawbacks of the MNT cycle.} There are two main drawbacks of the MNT cycle.
\begin{enumerate}
\item \emph{Sparseness of MNT curves that can be constructed in practice.} The MNT cycle works with the primes $p=x^2-x+1$ and $q=x^2+1$ for some $x \in \Z$, but the MNT curves $E/\F_p$ and $\hat{E}/\F_q$ can only be constructed if the CM discriminant $D$, the squarefree part of $3x^2-2x+3$, is \emph{small} enough (say, less than $10^{17}$). Of all the $x$ values that correspond to $p$ and $q$ being prime, those that also correspond to a small discriminant $D$ are rather rare. As $x$ grows large, the probability that $3x^2-2x+3$ happens to be divisible be a square that is large enough to make $D<10^{17}$ becomes exponentially small. This is why MNT curves are constructed using special values of $x$, those which are found as the solutions of Pell equations. Every candidate value for $D$ gives rise to a new Pell equation, the solution of which can be used as a candidate $x$ value; if such an $x$ also corresponds to $p$ and $q$ being prime, then the CM method can construct the MNT curves $E/\F_p$ and $\hat{E}/\F_q$. The main problem, however, is that the solutions to Pell equations are extremely large in general. The chance of finding a solution $x$ that also happens to be the right size at a target security level is unlikely; finding a solution that additionally corresponds to primes $p$ and $q$ with properties that are desirable in the context of SNARKs (like efficient field arithmetic and high $2$-adicity) is practically out of the question~\cite[\S 3.2]{scalable}.

\item \emph{Small embedding degrees.}  The two curves in the MNT cycle have embedding degrees that are too small to balance the elliptic curve discrete logarithm problem (ECDLP) and the finite field discrete logarithm problem (DLP) at moderate levels of security. Thus, instantiating these cycles forces $E/\F_p$ and $\hat{E}/\F_q$ to be defined over finite fields whose sizes are \emph{much} larger than the sizes of fields that non-pairing-friendly curves could be defined over. The subsequent performance inefficiencies are the main reason why cycles of pairing-friendly curves have not seen widespread deployment in pairing-based proof systems. Even the Mina protocol~\cite{mina}, which was originally founded on a cycle of MNT curves, has since abandoned pairing-based proof systems in favor of a non-pairing-friendly cycle. 

The absence of optimal pairing-friendly \emph{cycles} has led many SNARK designers to instead opt for \emph{chains} of pairing-friendly curves~\cite{gepetto,DBLP:conf/eurocrypt/HousniG22,AHG22}. Such chains allow for composition of pairing-based proofs, but the number of times proofs can be composed recursively is strictly less than the number of curves in the chain. For example, most chains found in the literature are 2-chains~\cite{ElHousni20,DBLP:conf/eurocrypt/HousniG22}, which only allow for one round of proof composition. Again, $2$-chains may suffice for specific target applications, but such applications are a far cry from the sorts of scalable pairing-based proof systems that unbounded recursive composition would support~\cite{scalable}. 

\end{enumerate}

\paragraph{Our contributions.} This paper presents two main contributions that are relevant to the two MNT drawbacks mentioned above. 
\begin{enumerate}
\item \emph{New cycles of supersingular elliptic curves.} In Section~\ref{sec:superMNT} we show how cycles can be instantiated for \emph{all} $x \in \Z$ that give $p=x^2-x+1$ and $q=x^2+1$ as primes, irrespective of the CM discriminant of the MNT curves. Rather than constructing the ordinary MNT cycle of $E$ and $\hat{E}$ over $\F_{p}$ and $\F_q$ via the CM method, we construct a cycle of supersingular curves $\EC$ and $\hat{\EC}$ over extension fields using Br{\"o}ker's algorithm~\cite{broker}. There are four possible combinations of extension degrees that arise in our construction (depending on the value of $x \bmod{12}$ -- see Theorem~\ref{thm:main}), the optimal case of which is when $\EC/\F_{p^2}$ forms a cycle with  $\hat{\EC}/\F_{q^2}$, which we depict in Figure~\ref{fig:10mod12}. 

\begin{figure}[!ht]
		\centering
		\begin{tikzpicture}
	            \node (A) at (1,1) {$E(\F_p)$};
		   \node (B) at (3,1) {$\EC(\F_{p^2})$};
		   \node (C) at (5,1) {$\F_{p^4}^\times$};
	            \draw[gray, thin] [dashed] (A) to (B);
	            \draw[gray, thin] [dashed] (B) to (C);
	            \node (D) at (1,-1) {$\hat{E}(\F_q)$};
	            \node (E) at (3,-1) {$\hat{\EC}(\F_{q^2})$};
	            \node (F) at (7,-1) {$\F_{q^6}^\times$};
	            \draw[gray, thin] [dashed] (D) to (E);
	            \draw[gray, thin] [dashed] (E) to (F);
	            \draw[->, bend left=30] (A) to (D);
	            \draw[->, bend left=30] (D) to (A);
	            \draw[->, bend left=30] (B) to (E);
	            \draw[->, bend left=30] (E) to (B);
	        \end{tikzpicture}
	\caption{The ordinary MNT cycle vs. the supersingular cycle.}
	\label{fig:10mod12}

\end{figure}
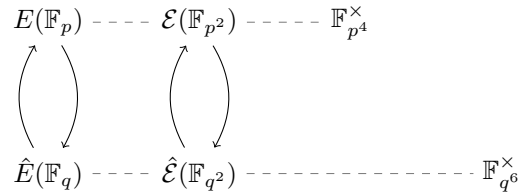

Both the MNT cycle and the supersingular cycle in Figure~\ref{fig:10mod12} correspond to pairings for which the target groups $\mathbb{G}_T$ and $\hat{\mathbb{G}}_T$ are the same, i.e. $\F_{p^4}^\times$ and $\F_{q^6}^\times$, respectively. Moreover, the ECDLP problems in both cycles are of order-$p$ and of order-$q$, meaning that the overall security picture of both constructions is, to the best of current knowledge, equivalent. The main efficiency difference between the two constructions arises from the supersingular curves being defined over $\F_{p^2}$ and $\F_{q^2}$. This means that elements in the pairing groups $\mathbb{G}_1$ and $\hat{\mathbb{G}}_1$ will be twice as large, and that computations with these elements will be less efficient (note that there is no such penalty in $\mathbb{G}_2$ and $\hat{\mathbb{G}}_2$). 

The benefit of the supersingular construction is that we get an infinite family of pairing-friendly cycles, both in theory and in practice. Any $x \in \Z$ that gives $p=x^2-x+1$ and $q=x^2+1$ as primes gives rise to a cycle of supersingular curves that can be constructed and used in recursive proof systems. On the other hand, the MNT family of cycles is only infinite in theory; ordinary MNT curves \emph{exist} for all such $x$ above, but only a finite number of them can be constructed via the CM method. 

Overcoming the sparseness of MNT cycles unlocks new possibilities for recursive proof systems. The main possibility we explore in this paper is that of constructing lollipops of pairing-friendly elliptic curves, which are designed to partially overcome the second drawback of the MNT cycle. 

\item \emph{Lollipops of pairing-friendly curves.} One possibility that is unlocked by the supersingular cycles in this paper is the explicit construction of \emph{lollipops} of pairing-friendly elliptic curves, the notion of which was first conceived by Meckler in 2019~\cite{meckler}. In order to address the inefficiencies arising from the small embedding degrees in the MNT cycle (see above), he proposed constructing a chain of pairing-friendly curves that eventually leads into the MNT cycle. The idea is that the pairing-friendly curves in the chain can be defined over smaller fields and thus give rise to more efficient computations in the initial layer(s) of the proof system, while (beyond those layers) the MNT cycle still affords unbounded layers of recursive proof composition.

  \begin{figure}[!ht]
        \centering
        \subfloat[$2$-cycle\label{fig:2cycles}]{%
                \parbox[c][2cm][b]{0.3\linewidth}{ %% Adjust alignment
                        \centering
                        \begin{tikzpicture}
                                \node (A) at (0,1) {$\bullet$};
                                \node (B) at (0,0) {$\bullet$};
                                \draw[->, bend left=60] (A) to (B);
                                \draw[->, bend left=60] (B) to (A);
                        \end{tikzpicture}
                }
        }\hfil
	\subfloat[$2$-chain\label{fig:2chain}]{%
                \parbox[c][2cm][b]{0.3\linewidth}{%
                        \centering   
                        \begin{tikzpicture}
                                \node (A) at (0,1) {$\bullet$};
                                \node (B) at (0,0) {$\bullet$};
                                \draw[->] (B) to (A);
                        \end{tikzpicture}
                }
        }\hfil
	\subfloat[$(2,2)$-lollipop\label{fig:22lollipop}]{%
                \parbox[c][2cm][b]{0.3\linewidth}{%
                        \centering
                        \begin{tikzpicture}
                                \node (A) at (0,2) {$\bullet$};
                                \node (B) at (0,1) {$\bullet$};
                                \node (C) at (0,0) {$\bullet$};
                                \draw[->, bend left=60] (A) to (B);
                                \draw[->, bend left=60] (B) to (A);
                                \draw[->] (C) to (B);
                        \end{tikzpicture}
                }
        }
        \caption{A $2$-cycle, a $2$-chain, and their combination: a $(2,2)$-lollipop.}
	\label{fig:all3}
    \end{figure}
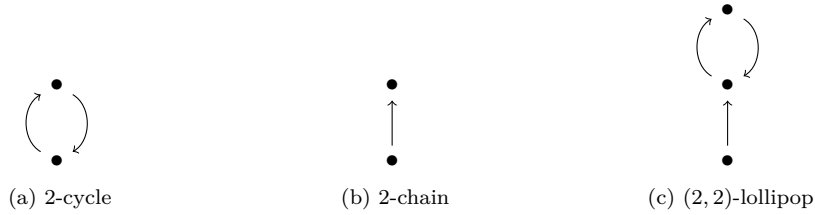

In Section~\ref{sec:constructing} we give a construction of pairing-friendly $(2,2)$-lollipops\footnote{In general, a $(m,n)$-lollipop would have $m$ curves in the ``stick'' and $n$ curves in the cycle (one curve is common to both the stick and the cycle), but we will focus on $(2,2)$-lollipops in this paper.}, which are $2$-chains that are connected to $2$-cycles. In Figure~\ref{fig:all3} we depict the differences between these three possibilities: a `$\bullet$' represents a pairing-friendly curve, and (here and throughout the paper) an arrow $\rightarrow$ pointing from $A$ to $B$ means $B$ is pairing-friendly with respect to the (characteristic of the) field of definition of $A$.  

It is worth noting that, in 2019, as part of the ``Coda+Dekrypt SNARK challenge''~\cite{coinlist}, the founders of the Coda protocol~\cite{coda} (now called the Mina protocol\footnote{Mina~\cite{mina} uses recursive composition of pairing-based proofs so that ``Users only need to check a singular, recursive `Proof of Everything'".}) offered a cash prize of \$20,000 USD for the highest quality submission of pairing-friendly lollipops. The Coda+Dekrypt SNARK challenge has since expired, but to the best of our knowledge the call for lollipops went unanswered. In Section~\ref{sec:examples} we present the 18 lollipops we found using the method we describe in Section~\ref{sec:constructing}. These lollipops offer between 80 and 128 bits of security and, to our knowledge, are the only pairing-friendly lollipops that have been found to date. 
\end{enumerate} 

\paragraph{Future work.} The present article is entirely constructive, and as such we do not make any concrete performance comparisons between the supersingular and MNT cycles. Future work will aim to instantiate and implement the optimal supersingular construction from Section~\ref{sec:superMNT}, in order to benchmark the performance trade-offs between the two cycles. As we mentioned above, the sole drawback of the supersingular cycles is that the source group $\G_1$ is defined over $\F_{p^2}$ and the source group $\hat{\G}_1$ is defined over $\F_{q^2}$; this will not only affect the performance of operations in $\G_1$ and $\hat{\G}_1$, but also the proof sizes themselves. On the other hand, MNT curves are limited in the optimisations they can exploit from the literature, both in the elliptic curve groups $\G_1$, $\G_2$, $\hat{\G}_1$ and $\hat{\G}_2$, and in the pairing computations. Our optimal construction typically finds $\EC/\F_{p^2}$ with $j(\EC)=1728$ and $\hat{\EC}/\F_{q^2}$ with $j(\hat{\EC})=0$, meaning that twists of degree larger than 2 can be used to compress and accelerate computations in both $\G_2$ and $\hat{\G}_2$. This, combined with the flexibility of choosing more performant field arithmetic in our construction, gives us hope that we can make up for the $\G_1$ and $\hat{\G}_1$ drawbacks elsewhere in a SNARK implementation. Finally, it is worth mentioning that the theoretical minimum of the number of group elements per proof is two~\cite{Groth16}, and that the only known way of achieving this lower bound is in the setting of supersingular pairings.

\section{Preliminaries}\label{sec:prelims}

For $n \in \Z_{>0}$, the $n$-th \emph{cyclotomic polynomial} $\Phi_n(x)$ is given by 
\[\Phi_n(x) = \dfrac{x^n-1}{\prod_{\substack{d\mid n \\ d<n}}\Phi_d(x)}.\]
Throughout this paper we will commonly make use of the following lemma.
\begin{lemma}[{\cite[Lemma 2.9]{WashingtonCyclotomic}}]\label{order-lemma}
    Let $p$ be a prime and $n,a\in \Z_{>0}$ such that $p\nmid na$. Then $\ord_p(a)=n$ iff $p\mid \Phi_n(a)$.
\end{lemma}

\paragraph{Pairing-friendly elliptic curves.} For an extensive survey on pairing-friendly curves, we refer to~\cite{taxonomy}. Let $E/\F_{p^u}$ be an elliptic curve and $r$ be a large prime such that $r \mid \#E(\F_{p^u})$. Then $E$ is said to have \emph{embedding degree} $k$ (with respect to $r$) if $\ord_r(p^u)=k$. Since $r \gg k$, we can use Lemma~\ref{order-lemma} to deduce that $E$ has embedding degree $k$ iff $r \mid \Phi_k(p^u)$. If $k$ is \emph{small} enough, e.g. $k \leq 50$, then $E$ is said to be \emph{pairing-friendly}. All pairing-friendly curves in this paper have $k \leq 6$. 

\paragraph{The CM method.} Hasse's theorem~\cite{hasse1936theorie} states that the number of points on an elliptic curve $E/\F_{p^u}$ is $\#E(\F_{p^u})=p^u+1-t$, where the \emph{trace of Frobenius} $t$ is bounded by $|t| \leq 2p^{u/2}$. On input of a given $t$ within the Hasse interval, we can construct a curve (i.e. compute its coefficients) with $p^u+1-t$ rational points via the \emph{complex multiplication} (CM) method~\cite[\S 4]{Morain91}. Write 
\begin{align}\label{eq:CM}
DV^2=4p^u-t^2,
\end{align}
where $D, V \in \Z$, and where $D$ is \emph{squarefree}. The CM method finds $E$ by computing the \emph{Hilbert class polynomial} $H_{D}(X) \in \F_{p^u}[X]$, the roots of which correspond to $j$-invariants of elliptic curves whose CM \emph{discriminant} is $D$. If $j \in \F_{p^u}$ is such that $H_D(j)=0$, then we can write $E$ as (a quadratic twist of) $E \colon y^2=x^3+ax-a$, where $a=-27j/(4(j-1728))$; the only exceptions are $j=0$, in which case we take $E$ as (a quartic twist of) $E \colon y^2=x^3+1$, and $j=1728$, in which case we take $E$ as (a sextic twist of) $E \colon y^2=x^3+x$. For details concerning twists, see~\cite[Proposition X.5.4]{silverman2009arithmetic}.

When $u=1$, elliptic curves \emph{exist} for all $t$ with $|t| \leq 2\sqrt{p}$. When $u>1$, there are fewer than $u$ values of $t$ which do not correspond to an elliptic curve~\cite[Theorem 4.2]{Schoof87}. Either way, when $p^u$ is of cryptographic size, only a tiny fraction of these $t$ values correspond to curves that can actually be constructed via the CM method. This is because the computation of the Hilbert class polynomial becomes infeasible if $D$ is large. The time complexity of the best known algorithm for computing $H_D(X)$ is in $\tilde{O}(|D|)$~\cite{Sutherland12}, and current record CM computations\footnote{See~\url{https://math.mit.edu/~drew/CMRecords.html}.} have $|D| < 10^{17}$. For $q$ of cryptographic size, most traces $t$ in the Hasse interval will correspond to a discriminant $D$ in~\eqref{eq:CM} that far exceeds those in these record computations. In the next subsection we will show how Pell equations can be used to find the special values of $t$ that \emph{do} correspond to sizes of $D$'s that make the CM method feasible. 

\paragraph{Finding MNT cycles with small CM discriminants.} Recall from Section~\ref{sec:intro} that $2$-cycles of MNT curves are defined by taking large primes 
\begin{align}\label{eq:MNTparams}
p=\Phi_6(x)=x^2-x+1 \quad \quad {\rm and} \quad\quad q=\Phi_4(x)=x^2+1 
\end{align}
 for some $x \in \Z$. The MNT curve $E/\F_p$ has trace $t_E=-x+1$, which gives $\#E(\F_p)=p+1-t_E = x^2+1=q$, while the MNT curve $\hat{E}/\F_q$ has trace $t_{\hat{E}}=x+1$, which gives $\#\hat{E}(\F_q)=q+1-t_{\hat{E}} = x^2-x+1=p$. It follows from~\eqref{eq:MNTparams} that $\ord_{q}(p)=4$ and $\ord_{p}(q)=6$, so $E$ has embedding degree 4 and $\hat{E}$ has embedding degree 6. In both cases, substitution into~\eqref{eq:CM} yields 
\begin{align}\label{eq:CMMNT}
DV^2=3x^2-2x+3. 
\end{align}
Over all instances of $x \in \Z$ that correspond to $p$ and $q$ in~\eqref{eq:MNTparams} being prime, we expect that the vast majority give rise to $D \approx p \approx q$. Put another way, if $x$ is chosen at random from a large interval, we cannot expect the value of $3x^2-2x+3$ from~\eqref{eq:CMMNT} to contain a large square factor $V$, and thus $D = O(x^2)$ in most cases. 

The way MNT curves are constructed is to instead compute the few values of $x$ that correspond to small values of $D$. Putting $U=3x-1$ into~\eqref{eq:CMMNT} yields the \emph{generalised Pell equation}
\begin{align}\label{eq:Pell}
U^2-3DV^2 &= -8.
\end{align}
Solving~\eqref{eq:Pell} for a small, fixed value of $D$ yields the pair $(U,V) \in \Z^2$. If $U \in 2+3\Z$, then we can take $x=(U+1)/3$ and proceed by checking if $p=x^2-x+1$ and $q=x^2+1$ are prime. If they are, then we can use the CM method to construct the MNT curves $E/\F_p$ and $\hat{E}/\F_q$. 

The reason we must solve many (e.g. millions of) Pell equations to find suitable MNT parameters is that the solutions $(U,V)$ to~\eqref{eq:Pell} are unlikely to be the size we want at a given security level. Even in the cases where the solutions are of the right size, it is unlikely that both $U \in 2+3\Z$ and the corresponding values of $p$ and $q$ are both prime. 

In the next section we show that cycles can still be constructed over fields of characteristic $p$ and $q$ regardless of the CM discriminant $D$, by instead exploring a supersingular construction.

\section{New cycles of supersingular elliptic curves}\label{sec:superMNT}

Recall from~\eqref{eq:MNTparams} that the MNT cycle has 
\begin{align*}
p=x^2-x+1 \quad \quad {\rm and} \quad\quad q=x^2+1,
\end{align*}
and is such that 
\begin{align}\label{eq:ordpq}
{\rm ord}_q(p)=4 \quad \quad {\rm and} \quad\quad {\rm ord}_p(q)=6.
\end{align}

The ordinary curves $E/\F_p$ and $\hat{E}/\F_q$ are such that $\#E(\F_p)=q$ and $\#\hat{E}(\F_q)=p$; they can only be constructed by the CM method for those values of $x$ which give rise to a small enough $D$ in~\eqref{eq:CMMNT}. In this section we show an alternative cycle construction over $\overline{\F}_p$ and $\overline{\F}_q$ that uses supersingular curves. These supersingular curves have $j$-invariants in $\F_p$ and $\F_q$, but the disadvantage of invoking the supersingular construction is that the respective $q$- and $p$-torsion points (that define the cycle) only become rational over extension fields. The crucial advantage, however, is that this construction works for \emph{all} values of $x$ that give rise to $p$ and $q$ as primes.

Cycles of supersingular elliptic curves were recently given in~\cite{SantosCN24}, but the cycles we introduce here have an important difference. Those in~\cite{SantosCN24} set $q \equiv 1 \bmod{p}$ so that any curve defined over $\F_q$ is forced to have embedding degree 1 with respect to $p$. Furthermore, the construction in~\cite{SantosCN24} also forced the bitlength of $q$ to always be (at least) twice the bitlength of $p$. In what follows we instead exploit the relationship between the special $p$ and $q$ in the MNT construction to avoid these impositions; our supersingular cycles have $p \approx q$ and allow for both curves to have embedding degrees greater than 1, as we see in Theorem~\ref{thm:main}.

We start by specialising the definition of a pairing-friendly $n$-cycle (see Section~\ref{sec:constructing}) to the case of $n=2$. 

\begin{definition}[Pairing-friendly $2$-cycle]\label{def:cycle}
We say that two elliptic curves $\EC/{\F_{p^u}}$ and $\hat{\EC}/{\F_{q^v}}$ are a \emph{pairing-friendly $2$-cycle}, denoted 
\[\EC \rightleftharpoons \hat{\EC},\]
if
\begin{enumerate}[label = (\roman*)]
\item $p \mid \#\hat{\EC}(\F_{q^v})$; \label{dec:pdiv}
\item $q \mid \#\EC(\F_{p^u})$; \label{dec:qdiv} 
\item $\EC$ is pairing-friendly with respect to $q$; and \label{def:pfq}
\item $\hat{\EC}$ is pairing-friendly with respect to $p$. \label{def:pfp}
\end{enumerate}
\end{definition}

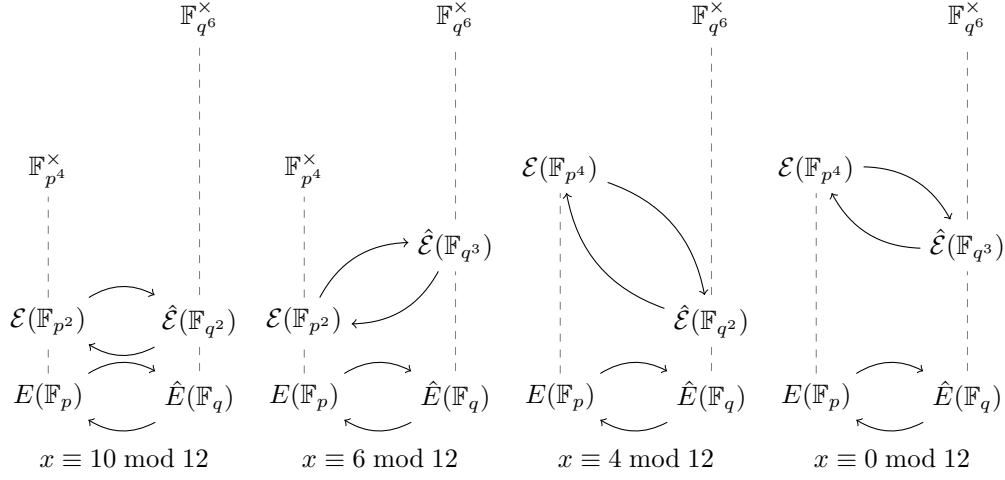
\begin{figure}[!ht]
	\begin{minipage}[b]{0.24\linewidth}
		\centering
		\begin{tikzpicture}
	            \node (A) at (-1,0) {$E(\F_p)$};
		   \node (B) at (-1,1) {$\EC(\F_{p^2})$};
		   \node (C) at (-1,3) {$\F_{p^4}^\times$};
	            \draw[gray, thin] [dashed] (A) to (B);
	            \draw[gray, thin] [dashed] (B) to (C);
	            \node (D) at (1,0) {$\hat{E}(\F_q)$};
	            \node (E) at (1,1) {$\hat{\EC}(\F_{q^2})$};
	            \node (F) at (1,5) {$\F_{q^6}^\times$};
	            \draw[gray, thin] [dashed] (D) to (E);
	            \draw[gray, thin] [dashed] (E) to (F);
	            \draw[->, bend left=30] (A) to (D);
	            \draw[->, bend left=30] (D) to (A);
	            \draw[->, bend left=30] (B) to (E);
	            \draw[->, bend left=30] (E) to (B);
	        \end{tikzpicture}
		$ x \equiv 10 \bmod{12}$
	\end{minipage}
\begin{minipage}[b]{0.24\linewidth}
		\centering
		\begin{tikzpicture}
	            \node (A) at (-1,0) {$E(\F_p)$};
		   \node (B) at (-1,1) {$\EC(\F_{p^2})$};
		   \node (C) at (-1,3) {$\F_{p^4}^\times$};
	            \draw[gray, thin] [dashed] (A) to (B);
	            \draw[gray, thin] [dashed] (B) to (C);
	            \node (D) at (1,0) {$\hat{E}(\F_q)$};
	            \node (E) at (1,2) {$\hat{\EC}(\F_{q^3})$};
	            \node (F) at (1,5) {$\F_{q^6}^\times$};
	            \draw[gray, thin] [dashed] (D) to (E);
	            \draw[gray, thin] [dashed] (E) to (F);
	            \draw[->, bend left=30] (A) to (D);
	            \draw[->, bend left=30] (D) to (A);
	            \draw[->, bend left=30] (B) to (E);
	            \draw[->, bend left=30] (E) to (B);
	        \end{tikzpicture}
		$ x \equiv 6 \bmod{12}$
	\end{minipage}
	\begin{minipage}[b]{0.24\linewidth}
		\centering
		\begin{tikzpicture}
	            \node (A) at (-1,0) {$E(\F_p)$};
		   \node (B) at (-1,3) {$\EC(\F_{p^4})$};
	            \draw[gray, thin] [dashed] (A) to (B);
	            \node (D) at (1,0) {$\hat{E}(\F_q)$};
	            \node (E) at (1,1) {$\hat{\EC}(\F_{q^2})$};
	            \node (F) at (1,5) {$\F_{q^6}^\times$};
	            \draw[gray, thin] [dashed](D) to (E);
	            \draw[gray, thin] [dashed] (E) to (F);
	            \draw[->, bend left=30] (A) to (D);
	            \draw[->, bend left=30] (D) to (A);
	            \draw[->, bend left=30] (B) to (E);
	            \draw[->, bend left=30] (E) to (B);
	        \end{tikzpicture}
		$ x \equiv 4 \bmod{12}$
	\end{minipage}
	\begin{minipage}[b]{0.24\linewidth}
		\centering
		\begin{tikzpicture}
	            \node (A) at (-1,0) {$E(\F_p)$};
		   \node (B) at (-1,3) {$\EC(\F_{p^4})$};
	            \draw[gray, thin] [dashed] (A) to (B);
	            \node (D) at (1,0) {$\hat{E}(\F_q)$};
	            \node (E) at (1,2) {$\hat{\EC}(\F_{q^3})$};
	            \node (F) at (1,5) {$\F_{q^6}^\times$};
	            \draw[gray, thin] [dashed](D) to (E);
	            \draw[gray, thin] [dashed] (E) to (F);
	            \draw[->, bend left=30] (A) to (D);
	            \draw[->, bend left=30] (D) to (A);
	            \draw[->, bend left=30] (B) to (E);
	            \draw[->, bend left=30] (E) to (B);
	        \end{tikzpicture}
		$ x \equiv 0 \bmod{12}$
	\end{minipage}
	\caption{The ordinary MNT cycle vs. the four supersingular cycles in Theorem~\ref{thm:main}. Further explanation in text.}
	\label{fig:superMNT}
\end{figure}

In Figure~\ref{fig:superMNT} giving we give a depiction of the constructions that follow in Theorem~\ref{thm:main}. The difference in the four constructions is the field of definition of the curves $\EC$ and $\hat{\EC}$. These differences arise based on the values of $p \bmod{12}$ and $q \bmod{12}$, because Waterhouse's theorem~\cite[Theorem 4.1]{waterhouse} gives precise conditions on the (non-)existence of various supersingular curves that depends on these values. With $p=x^2-x+1$ and $q=x^2+1$, it turns out that (for all $x \neq 2$) there are only 4 classes of $x \in \Z/12\Z$ that can give both $p$ and $q$ prime; odd $x$ gives even $q$ and $x \in \{2 + 12 \Z, 8 +12 \Z\}$ gives $p$ a multiple of $3$.\footnote{We note for completeness that, when $x=2$, we still get a cycle with $p=3$ and $q=5$. There is a supersingular curve $\EC/\F_{3^2}$ with $j(\EC)=1728$ and $\#\EC(\F_{3^2})=2q$, and a curve $\hat{\EC}/\F_{q}$ with $j(\hat{\EC})=0$ and $\#\hat{\EC}(\F_{5})=2p$. Both curves have embedding degree $2$.} This is why both of the propositions below start with the statement of the classes of $x \in \Z/12\Z$.

\begin{theorem}\label{thm:main}
Let $x \in \Z$ be such that be such that $p = x^2-x+1$ and $q=x^2+1$ are prime. Then there exist two supersingular curves 
\[\EC/\F_{p^u} \quad\quad { and} \quad\quad \hat{\EC}/\F_{q^v}\]
with $u \in \{2,4\}$ and $v \in \{2,3\}$ such that
\[\EC \rightleftharpoons \hat{\EC},\]
and such that the target group of order-$q$ Weil pairing on $\EC$ is $\F_{p^4}^\times$, and such that the order-$p$ Weil pairing on $\hat{\EC}$ is $\F_{q^6}^\times$.  
Moreover, the four possibilities for $(u,v)$ arising from $u \in \{2,4\}$ and $v \in \{2,3\}$ correspond to the four possible classes of $x \in \Z/12\Z$ leading to prime numbers $p$ and $q$. In particular
\begin{enumerate}
\item if $x\equiv 10\bmod {12}$ then $(u,v) = (2,2)$; 
\item if $x\equiv 6\bmod {12}$ then $(u,v) = (2,3)$;
\item if $x\equiv 4\bmod {12}$ then $(u,v) = (4,2)$; and 
\item if $x\equiv 0\bmod {12}$ then $(u,v) = (4,3)$.
\end{enumerate}
 
\end{theorem}
\begin{proof}
We prove each of the requirements in Definition~\ref{def:cycle} for each case. Note that applying Lemma~\ref{order-lemma} to equation~\eqref{eq:ordpq} implies that $p\mid \Phi_6(q)$ and $q\mid \Phi_4(p)$. Next, we observe that:
\begin{enumerate}[label = {(\alph*)}]
    \item If $p \equiv 7\mod 12$, then by~\cite[Theorem 4.1(5)(ii)]{waterhouse}  there exists a supersingular curve $\EC/\F_{p^2}$ with $\#\EC(\F_{p^2}) = p^2+1$. Then $q\mid \#\EC(\F_{p^2})$ because $\#\EC(\F_{p^2})=\Phi_4(p)$ and $q \mid \Phi_4(p)$. \label{case:p7}
    \item If $q\equiv 5\mod 12$, then by~\cite[Theorem 4.1(3)]{waterhouse} there exists $\hat{\EC}/\F_{q^2}$ with $q^2+1-q$ rational points. Then $p\mid \#\hat{\EC}(\F_{q^2})$ because $\#\hat{\EC}(\F_{q^2})=\Phi_6(q)$ and $p \mid \Phi_6(q)$. \label{case:q5}
    \item If $q\equiv 1\mod 12$, then by~\cite[Theorem 4.1(5)(i)]{waterhouse} there exists $\hat{\EC}/\F_{q^3}$ with $q^3+1$ rational points. Then $p\mid \#\hat{\EC}(\F_{q^3})$ because $\#\hat{\EC}(\F_{q^3})=(q+1)\Phi_6(q)$ and $p \mid \Phi_6(q)$. \label{case:q1}
    \item If $p \equiv 1\mod 12$, then by~\cite[Theorem 4.1(2)]{waterhouse} there exists a supersingular curve $\EC/\F_{p^4}$ with $\#\EC(\F_{p^4}) = p^4+1+2p^2$. Then $q\mid \#\EC(\F_{p^4})$ because $\#\EC(\F_{p^4})=(\Phi_4(p))^2$ and $q \mid \Phi_4(p)$. \label{case:p1}
\end{enumerate}
We can now prove~\ref{dec:pdiv} and~\ref{dec:qdiv} of Definition~\ref{def:cycle} for all four cases.
\begin{enumerate}
    \item Here $p \equiv 7\mod 12$ and $q\equiv 5\mod 12$, hence the proof follows from \ref{case:p7} and \ref{case:q5}.
    \item Here $p \equiv 7\mod 12$ and $q\equiv 1\mod 12$, hence the proof follows from \ref{case:p7} and \ref{case:q1}.
    \item Here $p \equiv 1\mod 12$ and $q\equiv 5\mod 12$, hence the proof follows from \ref{case:p1} and \ref{case:q5}.
    \item Here $p \equiv 1\mod 12$ and $q\equiv 1\mod 12$, hence the proof follows from \ref{case:p1} and \ref{case:q1}.
\end{enumerate}
Finally, the embedding degrees follow directly from~\cite[Theorem IX.20]{GalbraithBSS}, which also proves~\ref{def:pfq} and~\ref{def:pfp} of Definition~\ref{def:cycle} for all four cases.
\end{proof}

\paragraph{Optimality of Theorem~\ref{thm:main}.} We now present Remark~\ref{rem:q} and Corollary~\ref{coll}, which shows that Theorem~\ref{thm:main} is optimal in the following sense: if $p$ and $q$ are parameterized as in~\eqref{eq:ordpq}, then there is no $2$-cycle of supersingular elliptic curves that allows us to work over higher degree extension fields than those presented in Theorem~\ref{thm:main}. 

One might hope to find a supersingular cycle $\EC/\F_{p^u}$ and $\hat{\EC}/\F_{q^v}$ with, say, $u>4$ and $v>6$, for which the target groups of the order-$q$ and order-$p$ Weil pairings are larger than $\F_{p^4}^\times$ and $\F_{q^6}^\times$; this would allow more balanced ECDLP and DLP complexities so that the sizes of $p$ and $q$ can be smaller. Unfortunately, it turns out that this hope is a lost cause; in Remark~\ref{rem:q} we give five examples of fields larger than $\F_{p^4}$ or $\F_{q^3}$ that could have been used as the fields of definition in Theorem~\ref{thm:main}, each time showing that none of the resulting constructions will have DLP security governed by fields larger than $\F_{p^4}^\times$ and $\F_{q^6}^\times$. In Corollary~\ref{coll} we prove that the trend observed in these examples is true in general. 

\begin{remark}\label{rem:q}
    We note a handful of example constructions we tried in the hope of reducing the sizes of $p$ and $q$, before formulating Corollary~\ref{coll} below. 
    \begin{enumerate}[label={(\Roman*)}]
        \item For $q\equiv 1,5\mod 12$, there exists $\hat{\EC}/\F_{q^6}$ with $p\mid \#\hat{\EC}(\F_{q^6}) = q^6 + 1 + 2q^3 = (q+1)^2\Phi_6(q)^2$~\cite[Theorem 4.1(2)]{waterhouse}. Moreover, $\hat{\EC}[p]\subseteq \hat{\EC}(\F_{q^6})$~\cite[Theorem IX.20]{GalbraithBSS}. However, we did not choose this curve for \ref{case:q5} or \ref{case:q1} in the proof of Theorem~\ref{thm:main} because it will offer the same level of security as the one we considered over lower degree field extension. \label{extra1}
        \item For $q\equiv 1,5\mod 12$, there exists $\hat{\EC}/\F_{q^{12}}$ with $p\mid \#\hat{\EC}(\F_{q^{12}}) = q^{12} + 1 - 2q^6 = (q^6-1)^2=(q^2-1)^2\Phi_6(-q)^2\Phi_6(q)^2$~\cite[Theorem 4.1(2)]{waterhouse}. Moreover, $\hat{\EC}[p]\subseteq \hat{\EC}(\F_{q^{12}})$~\cite[Theorem IX.20]{GalbraithBSS}. However, it is $\overline{\F}_{q}$-isogenous to~\ref{extra1} above.
        \item For $q\equiv 5\mod 12$, there exists $\hat{\EC}/\F_{q^4}$ with $p\mid \#\hat{\EC}(\F_{q^4}) = q^4 + 1 + q^2 = \Phi_6(-q)\Phi_6(q)$~\cite[Theorem 4.1(3)]{waterhouse}. Moreover, $\hat{\EC}[p]\subseteq \hat{\EC}(\F_{q^{12}})$~\cite[Theorem IX.20]{GalbraithBSS}. However, it is $\overline{\F}_{q}$-isogenous to~\ref{case:q5} in the proof of Theorem~\ref{thm:main}.
	\item For $p\equiv 1,7\mod 12$, there exists $\EC/\F_{p^{8}}$ with $q\mid \#\EC(\F_{p^8}) = p^8 + 1 - 2p^4 = (p^4-1)^2=(p^2-1)^2\Phi_4(p)^2$~\cite[Theorem 4.1(2)]{waterhouse}. Moreover, $\EC[q]\subseteq \EC(\F_{p^{8}})$~\cite[Theorem IX.20]{GalbraithBSS}. However, it is $\overline{\F}_{p}$-isogenous to~\ref{case:p1} in the proof of Theorem~\ref{thm:main}.
	\item     We can use~\cite[Theorem 4.1]{waterhouse} and~\cite[Proposition 3.7]{Schoof87} to show existence of a supersingular cycle consisting of $\EC/\F_{p^{10}}$ with $q\mid \#\EC(\F_{p^{10}}) = p^{10} + 1 = \Phi_4(p)(p^8-p^6+p^4-p^2+1)$ and $\EC[q] \subseteq \EC(\F_{p^{20}})$; and $\hat{\EC}/\F_{q^{9}}$ with $p\mid \hat{\EC}(\F_{q^{9}}) = q^9+1 = \Phi_6(q)(q+1)(q^6-q^3+1)$ and $\hat{\EC}[p] \subseteq \hat{\EC}(\F_{q^{18}})$. However, such a cycle will offer the same level of security as the cycle of isogenous curves $\EC'/\F_{p^{2}}$ with $q\mid \#\EC'(\F_{p^{2}}) = p^{2} + 1 = \Phi_4(p)$ and $\EC'[q] \subseteq \EC'(\F_{p^{4}})$; and $\hat{\EC'}/\F_{q^{3}}$ with $p\mid \hat{\EC'}(\F_{q^{3}}) = q^3+1 = \Phi_6(q)(q+1)$  and $\hat{\EC'}[p] \subseteq \hat{\EC'}(\F_{q^{6}})$.
    \end{enumerate} 
\end{remark}

\begin{corollary}\label{coll}
    Let $p=\Phi_6(x)$ and $q=\Phi_4(x)$ be primes with and suppose $\EC/\F_{p^u}$ and $\hat{\EC}/\F_{q^v}$ are a cycle of supersingular elliptic curves. Then there exists a supersingular cycle of isogenous curves $\EC'/\F_{p^4}$ and $\hat{\EC'}/\F_{q^6}$ such that the torsion points $\hat{\EC'}[p]\subseteq \hat{\EC'}(\F_{q^6})$ and $\EC'[q] \subseteq \EC'(\F_{p^4})$.
\end{corollary}
\begin{proof}
    In general, since $\EC$ (resp. $\hat{\EC}$) is supersingular with $\#\EC(\F_{p^u})=p^u+1-t$ (resp. $\#\hat{\EC}(\F_{q^v})=q^v+1- t$ ) and $t^2\in\{0,p^u,4p^u\}$ (resp. $t^2\in\{0,q^v,4q^v\}$), we have $\EC[q] \subseteq \EC(\F_{p^{uk}})$ (resp. $\hat{\EC}[p] \subseteq \hat{\EC}(\F_{q^{vk}})$) with $k=1$ if $t^2=4p^u$ (resp. $t^2=4q^v$), $k=2$ if $t^2=0$, or $k=3$ if $t^2=p^u$ (resp. $t^2=q^v$)~\cite[Table 1]{MOV}. The above theorem lists the minimal possible values of $u$ and $v$ for each possible value of $x$ leading to a cycle. 
\end{proof}

\paragraph{Supersingular vs. MNT security.} There is no known difference in the security picture of our supersingular cycles and the underlying MNT cycles. Both cycles can use pairings to map ECDLP instances to the same order-$p$ subgroup of $\F_{q^6}^\times$ and/or the same order-$q$ subgroup of $\F_{p^4}^\times$, so the finite field DLP's are identical in both cases. Due to the large fields in the supersingular scenario, there may be a small concrete difference in solving the ECDLP's directly via Pollard's rho algorithm, but a real-world attack would never target the ECDLP's; the small embedding degrees always ensure that the finite field DLPs are the weak point of the cycle itself. The point of the remaining sections of this paper is to find another pairing-friendly curve, $E$, whose defining field and embedding degree (and thus DLP complexity) is the same as one of the curves in the cycle, but for which the complexity of solving the ECDLP in $E[r]$ is much closer to the complexity of solving the DLP in $\F_{p^4}^\times$. 

\paragraph{Br{\"o}ker's algorithm.} Over a field $\F_{p^u}$, the algorithm that constructs supersingular curves of a given trace (i.e. group order) is due to Br{\"o}ker~\cite{broker}. It starts by constructing a supersingular curve over the ground field, $\EC/\F_p$, and then outputs $\EC'/\F_{p^u}$ as an $\F_{p^u}$-twist of $\EC$. Since twists have the same $j$-invariant, it follows that the supersingular curves output by Br{\"o}ker's algorithm\footnote{In general, most supersingular curves have $j \not\in \F_{p}$, but they always come from the same isogeny class as a curve with $j \in \F_p$, which is what Br{\"o}ker's algorithm outputs.} always have $j(\EC') \in \F_p$, regardless of the field of definition of $\EC'$ -- see Table~\ref{tab:examples}. Br{\"o}ker's algorithm also constructs $\EC/\F_p$ by finding a root $j_0$ of the Hilbert class polynomial $H_\D(X) \in \F_p[X]$, but in the supersingular case the value of $\D$ that is used is the first prime $\D \equiv 3 \pmod{4}$ where $-\D$ is not a quadratic residue in $\F_p$. 

\begin{remark}[High-security, highly $2$-adic cycles]\label{rem:without}
Recall from Section~\ref{sec:intro} that a drawback of the MNT cycle is that it is infeasible to construct instances with large $2$-adicity, particularly at high-security levels~\cite{scalable,guillevic}. Using the supersingular construction allows us to simply search for values of $x$ with $2^\ell \mid x$ such that $p=x^2-x+1$ and $q=x^2+1$ are prime, noting that we immediately get a cycle where $2^\ell \mid p-1$ and $2^{2\ell} \mid q-1$. We ran a quick search for the largest $\ell$'s corresponding to 8 values of $x$ with $\lceil \log_2 x \rceil \in \{128,192,256,384,512,768,1024,2024\}$, and found the following values that give $p$ and $q$ primes: $x=2^{113}\cdot 32123$, $x=2^{176}\cdot 40335$, $x=2^{239}\cdot 108445$, $x=2^{370}\cdot 10431$, $x=2^{493}\cdot 354617$, $x=2^{748}\cdot 885549$, $x=2^{1006}\cdot 226419$, and $x=2^{2027}\cdot 1526763$. In light of Theorem~\ref{thm:main}, we note that searches restricting to $x \equiv 10 \bmod{12}$ would find optimal instances in the same time on average. 
\end{remark}

\section{Using supersingular cycles to construct pairing-friendly lollipops}\label{sec:constructing}

%%%%%%%%%%%%%%%%%

The main benefit of the construction in the previous section is that we no longer need to solve Pell equations to construct instances of cycles. In this section we show that reintroducing Pell equations allows us to find a pairing-friendly curve that is connected to our cycle, in order to produce the first construction of pairing-friendly lollipops.

We start by defining chains and cycles of pairing-friendly curves, pulling together definitions from El Housni and Guillevic~\cite[\S 2.3]{DBLP:conf/eurocrypt/HousniG22} and Chiesa, Chua and Weidner~\cite[Definition 7.1]{chiesa2019cycles}. 

\begin{definition}[Pairing-friendly chain]\label{def:chain}
An \emph{$m$-chain} of pairing-friendly elliptic curves is a list of distinct curves $E_1/\F_{p_1^{d_1}},\dots, E_m/\F_{p_m^{d_m}}$ with each $p_i$ a large prime, such that $p_i \mid \#E_{i+1}$ and $E_{i+1}$ is pairing-friendly with respect to $p_i$ for $i=1\dots m-1$. 
\end{definition}

\begin{definition}[Pairing-friendly cycle]\label{def:ncycle}
An \emph{$n$-cycle} of pairing-friendly elliptic curves is a list of distinct curves $E_1/\F_{q_1^{e_1}},\dots, E_n/\F_{q_n^{e_n}}$ with each $q_i$ a large prime, such that $q_i \mid \#E_{i+1}$ and $E_{i+1}$ is pairing-friendly with respect to $q_i$ for $i=1\dots n-1$, and such that $q_n \mid \#E_1$ and $E_1$ is pairing-friendly with respect to $q_n$. 
\end{definition}

We can now define lollipops by combining Definition~\ref{def:chain} and Definition~\ref{def:ncycle}. 

\begin{definition}[Pairing-friendly lollipop]\label{def:lollipop}
An \emph{(m,n)-lollipop} of pairing-friendly curves is an $m$-chain
\[
E_1 \rightarrow \dots \rightarrow E_{m} 
\]
of pairing-friendly curves, together with an $n$-cycle 
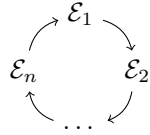
\begin{center}
\begin{tikzpicture}
            \node (A) at (0,1.5) {$\EC_{1}$};
            \node (B) at (0.75,0.75) {$\EC_{2}$};
            \node (C) at (0,0) {$\dots$};
	   \node (D) at (-0.75,0.75) {$\EC_{n}$};
            \draw[->, bend left=30] (A) to (B);
            \draw[->, bend left=30] (B) to (C);
	   \draw[->, bend left=30] (C) to (D);
	   \draw[->, bend left=30] (D) to (A);
\end{tikzpicture}
\end{center}
of pairing-friendly curves, such that $\{E_1,\dots E_m\} \cap \{\EC_{1},  \dots , \EC_{n} \}$ = $E_m$. 
\end{definition}

As stated in Section~\ref{sec:intro}, herein we will restrict to the case of $m=n=2$ and present $(2,2)$-lollipops.

Following Definition~\ref{def:lollipop}, and with the cycles we have defined in Section~\ref{sec:superMNT}, our task is to find another pairing-friendly curve defined over $\F_p$ with $p=x^2-x+1$ or $\F_q$ with $q=x^2+1$. The key to our construction is to observe that
\begin{align}\label{eq:N_q}
\Phi_4(p) = \underbrace{(x^2+1)}_{q} \cdot \underbrace{(x^2-2x+2)}_{N_q}
\end{align}
and
\begin{align}\label{eq:N_p}
 \Phi_6(q) = \underbrace{(x^2-x+1)}_{p} \cdot \underbrace{(x^2+x+1)}_{N_p}. 
\end{align}

Ideally, we want to find an ordinary curve $E/\F_p$ with $\#E(\F_p)=N_q$, so $E$ immediately has embedding degree 4 with respect to any large prime factor $r$ of $N_q$. Alternatively, we can also look for a curve $E/\F_q$ with $\#E(\F_q)=N_p$, so that $E$ has embedding degree 6 with respect to any large prime factor $r$ of $N_p$. As we discussed in Section~\ref{sec:prelims}, the feasibility of constructing these curves is related to the corresponding CM equations. We therefore examine the two cases separately. 

\paragraph{Case 1: $r$ divides $\#E(\F_p)=N_q$.} If $\#E(\F_p) = N_q$, then we must have that $t_E = x$, and hence the CM equation becomes 
\begin{align}\label{CMp}
DV^2 &= 4p-t^2, \nonumber\\
	&=4(x^2-x+1) - x^2, \nonumber\\
	&=3x^2-4x+4. 
\end{align}

\paragraph{Case 2: $r$ divides $\#E(\F_q)=N_p$.} If $\#E(\F_q) = N_p$, then we must have that $t_E = -x+1$, and hence the CM equation becomes 
\begin{align}\label{CMq}
DV^2 &= 4q-t^2, \nonumber\\
	&=4(x^2+1) - (-x+1)^2, \nonumber\\
	&=3x^2+2x+3. 
\end{align}

The next step is to make substitutions that transform these CM equations into generalised Pell equations. Taking $U=3x-2$ in~\eqref{CMp} or taking $U=3x+1$ in~\eqref{CMq}, we get the generalised Pell equation
\begin{align*}
U^2-3DV^2 &= -8, 
\end{align*}
which is the same as~\eqref{eq:Pell}. 

We first observe that any solution $U, V \in \Z$ to this generalised Pell equation must have $U \not\in 3\Z$. The substitutions $U=3x-2$ and $U=3x+1$ that gave rise to~\eqref{CMp} and~\eqref{CMq} reveal that the solutions we are interested in are those with $U \in 1 + 3\Z$. For any such solution, the integer $x=(U+2)/3$ will satisfy~\eqref{CMp} and the integer $x=(U-1)/3$ will satisfy~\eqref{CMq}. Given that we need $x$ to be even for $q=x^2+1$ to be prime, however, we see that only one of these two (consecutive) integers can be used as a candidate $x$-value for our lollipop construction. Nevertheless, it is convenient that we only need to write and launch one generalised Pell equation solver to search for solutions to both~\eqref{CMp} and~\eqref{CMq}. 

Before giving details of the full lollipop search we implemented, we first modify some results from~\cite{DBLP:conf/ants/KarabinaT08} in order to streamline the discriminants $D$ that we search over. 

\begin{lemma}\label{FqStickEp}
	Let $D$ be as in~\eqref{CMp} for $x\in 2\Z$. Then $3D\equiv 3,6,18,27 \pmod {48}$.
\end{lemma}
\begin{proof}
	Let $f(x) = 3x^2 - 4x + 4$ and $\delta(n)$ be the square-free part of an integer $n$. Then computing $3\cdot \delta(f(2m)) \mod {48}$ for $m\in \{0,1,\ldots,47\}$ leads to $\{3,6,18,27\}$.  
\end{proof}

\begin{lemma}\label{FpStickEq}
	Let $D$ be as in~\eqref{CMq} for $x\in 2\Z$. Then $3D\equiv 9 \pmod {24}$. 
\end{lemma}
\begin{proof}
	Let $f(x) = 3x^2 + 2x + 3$ and $\delta(n)$ be the square-free part of an integer $n$. Then computing $3\cdot\delta(f(2m)) \mod {24}$ for $m\in \{0,1,\ldots,23\}$ leads to $\{9\}$.  
\end{proof}

In what follows we set $D'=3D$. The above lemmas show we only need to search over 1/8 of the $D'$ values.

\paragraph{The high-level algorithm.} We summarise the above discussion by presenting the full lollipop algorithm as follows. On input of a small $D'$ (i.e. $D'<10^{17}$) as in Lemma~\ref{FqStickEp} or Lemma~\ref{FpStickEq}, do the following: 
\begin{enumerate}
\item \label{step:1} Solve the generalised Pell equation $U^2-D'V^2=-8$. 
\begin{enumerate}
\item If $D' \equiv 3,6,18,27 \pmod{48}$ and if $U \equiv 1 \pmod{3}$, then set $x = (U+2)/3$ and set $N=N_q=x^2-2x+2$ from~\eqref{eq:N_q}. Otherwise, pick a new $D'$ and start again. 
\item If $D' \equiv 9 \pmod{24}$ and if $U \equiv 1 \pmod{3}$, then set $x = (U-1)/3$ and set $N=N_p=x^2+x+1$ from~\eqref{eq:N_p}. Otherwise, pick a new $D'$ and start again.
\end{enumerate}
\item   \label{step:2} If $p=x^2-x+1$ and $q=x^2+1$ are both prime, proceed to Step~\ref{step:3}, otherwise pick a new $D'$ and start again. 
\item \label{step:3} Factor $N$, and let $r$ be a prime divisor of $N$ that is large enough to meet the requisite security level (e.g. $r \geq 2^{2\lambda}$ for $\lambda$-bit security). If no such $r$ exists, pick a new $D'$ and start again. 
\item \label{step:4} If $N=N_q$, then compute the Hilbert class polynomial $H_D(X) \in \F_p[X]$, otherwise if $N=N_p$ then compute the Hilbert class polynomial $H_D(X) \in \F_q[X]$. 
\item \label{step:5} Compute a root $j_0$ of $H_D(X)$ in $\F_p[X]$ (resp. $\F_q[X]$) and then construct the elliptic curve $E$ such that $j(E)=j_0$ (see Section~\ref{sec:prelims}).
\item \label{step:6} Output the lollipop as 
\[
E \rightarrow \EC \leftrightharpoons \hat{\EC}, 
\] 
where $\EC$ and $\hat{\EC}$ are the curves from Theorem~\ref{thm:main}.

\end{enumerate}

\paragraph{Practical adjustments.} The following are some practical adjustments that can be made in the first four steps of
our algorithm.
\begin{itemize}
	\item {\bf Step~\ref{step:1}}. Methods for solving the generalised Pell equation typically use variants of the continued fractions method, which allows us to bound the size of the solutions we want and abort if the solutions grow larger than this bound. We set a very large bound in an effort to try and find high-security instances, but any search for lollipops of a specific size would be accelerated significantly if the bound is lowered to reflect this. 
	\item {\bf Step~\ref{step:2}}. For the sake of efficiency, primality testing should be probabilistic within the full search algorithm and probable primes should be confirmed after the lollipops are output. 
	\item {\bf Step~\ref{step:3}}. For the sizes of lollipops we are searching for, this step is by far the most cumbersome. Here we are looking for one prime factor $r$ of a certain size, so full factorisations are not necessary, and in many cases it is wise to abort factorisations that do not terminate after a nominal amount of time (see below). 
	\item {\bf Step~\ref{step:4}}. For larger sizes of $D'$, there are other invariants analogous to the $j$-invariant for which the class polynomials are more efficient to compute and more compact to store, e.g. the Weber $f$-invariant. These alternative invariants also impose restrictions on $D'$, but are much preferred if they are compatible with those restrictions in Lemma~\ref{FqStickEp} and Lemma~\ref{FpStickEq}. More details can be found in Sutherland's {\sf classpoly} package\footnote{See~\url{https://math.mit.edu/~drew/classpoly.html}.}.
\end{itemize}

\section{Implementation and example lollipops}\label{sec:examples}

%%%%%%%%%%%%%%%%%%%%%%%%

In this section we give the implementation details of the algorithm in Section~\ref{sec:constructing} and present the 18 examples of pairing-friendly lollipops it found. 

\paragraph{Details of our search.} We now give the more fine-grained details of the search we ran, breaking down the discussion according to the six steps of the search algorithm. 

\begin{itemize}
\item {\bf Step~\ref{step:1}}. We used the \texttt{parallel GP} interface of \texttt{PARI/GP}~\cite{PARI2} running on AMD Ryzen Threadripper PRO 3995WX with 128 GB memory to solve the generalised Pell equation $U^2-D'V^2=-8$ using~\cite[Algorithm 1]{DBLP:conf/ants/KarabinaT08} and bounding $U$ at 2500 bits, for the following instances:

\begin{enumerate}[label=(\roman*)]
\item Equation~\eqref{CMp} with $D' \in \{48\ell+3, 48\ell+6, 48\ell+18, 48\ell+27\}$ as per Lemma~\ref{FqStickEp}. The equation was solved for $\ell=0$ to $\ell=1,295,124,911$, and solutions less than 2500 bits were obtained for $76,656,763$ of them. 
\item Equation~\eqref{CMq} with $D' \in 24\ell+9$ as per Lemma~\ref{FpStickEq}. The equation was solved for $\ell=0$ to $\ell=5,082,799,955$, and solutions less than 2500 bits were obtained for $63,146,643$ of them. 
\end{enumerate}
\item {\bf Step~\ref{step:2}}. We used the \texttt{parallel GP} interface of \texttt{PARI/GP} running on AMD Ryzen Threadripper PRO 3995WX with 128 GB memory to perform the following searches:
\begin{enumerate}[label=(\roman*)]
\item For the solutions of~\eqref{CMp} with $U \equiv 1\pmod{3}$, we set $x=(U+2)/3$ and tested the primality of $p=x^2-x+1$ and $q=x^2+1$, finding 4,281 pairs of primes $(p,q)$. 
\item For the solutions of~\eqref{CMq} with $U \equiv 1\pmod{3}$, we set $x=(U-1)/3$ and tested the primality of $p=x^2-x+1$ and $q=x^2+1$, finding 3,912 pairs of primes $(p,q)$.
\end{enumerate}
\item {\bf Step~\ref{step:3}}. We again used the \texttt{parallel GP} interface of \texttt{PARI/GP} running on AMD Ryzen Threadripper PRO 3995WX with 128 GB memory to perform the following factorisations (using \texttt{PARI/GP}'s internal factoring algorithms):
\begin{enumerate}[label=(\roman*)]
\item Of the 4,281 pairs of primes $(p,q)$ corresponding to~\eqref{CMp}, 350 candidates had both $p$ and $q$ at least 298 bits long.\footnote{We wanted a minimum of 80 bits of security, so chose the original {\sf MNT298} size from~\cite{scalable} as the lower bound on the sizes of $p$ and $q$.} Of these, we were able to successfully factor 186 of them. 
\item Of the 3,912 pairs of primes $(p,q)$ corresponding to~\eqref{CMq}, 337 candidates had both $p$ and $q$ at least 298 bits long. Of these, we were able to successfully factor 203 of them.
\end{enumerate}
\item {\bf Step~\ref{step:4}}. The following computations were run on an Intel Core i7-12800HX with 16 GB memory. Let $r$ be the biggest factor of $N \in \{N_p, N_q\}$ and recall that $N_p \approx N_q \approx p \approx q$. For all prime pairs where $(p,q)$ had $r > 2^{150}$ and where $N \gg r$ (we used a rough rule that ${\rm log}(N) \geq 2 {\rm log}(r)$, but made some exceptions), we computed the Hilbert class polynomial\footnote{This is computed using \texttt{classpoly d 0 q} where $d=-D$ if $-D\equiv 1\pmod 4$ or $d=-4D$ if $-D\equiv 2,3\pmod 4$.} in $\F_p[x]$ or $\F_q[x]$ using Sutherland's \texttt{classpoly} package~\cite{Sutherland11a,EngeSutherland}. 
\item {\bf Step~\ref{step:5}}. To find the roots of the polynomial $H_D(X)$ output from Step~\ref{step:4}, and then to construct the curve $E$ accordingly, we used \texttt{Magma}~\cite{Magma} running on Intel Core i7-12800HX with 16 GB memory. 
\item {\bf Step~\ref{step:6}}. We found a total of 18 lollipops $E \rightarrow \EC \leftrightharpoons \hat{\EC}$, 14 of which had the curve $E$ defined over $\F_p$ with embedding degree 4, and the other 4 of which had $E$ defined over $\F_q$ with embedding degree 6. These are all described in detail in the next section. 
\end{itemize}

%%%%%%%%%%%%%%%%%%%%%%%%

\paragraph{The examples.} We now present the 18 example of lollipops that were found in the search detailed above. We begin by recalling some notation from Section~\ref{sec:intro} and set some additional notation that is used in Figure~\ref{fig:types} and Table~\ref{tab:examples}. 
\iffullversion
The full list of detailed examples can be found in Appendix~\ref{sec:appendix}. Readers wishing to verify or experiment with any of the lollipops can refer to the \texttt{Magma} or \texttt{Pari/GP} code that is found at 
\begin{center}
\url{https://github.com/craigcostello/pairing-friendly-lollipops}.
\end{center}
\else
Readers wishing to verify or experiment with any of the lollipops can refer to the \texttt{Magma} or \texttt{Pari/GP} code that is found in the supplementary material accompanying this submission. 
\fi

\begin{itemize}
\item The examples are labelled as \textsf{lollipop-X-Y}, where ${\sf X}$ denotes the bitlength of $p$ and $q$, i.e. the characteristics of the fields of definition of the pairing-friendly curves in the lollipop, and ${\sf Y}$ denotes the bitlength of $r$, the characteristic of the non-pairing-friendly curve(s) attached to it. We also use ${\sf X}$ and ${\sf Y}$ as subscripts of curves to indicate the size of the underlying field characteristic. 
\item An arrow $A \rightarrow  B$ indicates that the curve $B$ is pairing-friendly with respect to the characteristic of the field of definition of the curve $A$.
\item $\EC_{\sf X}$ and $\hat{\EC}_{\sf X}$ are the supersingular (and thus pairing-friendly) curves in the cycle; $\EC_{\sf X}$ is defined over $\F_{p^2}$ or $\F_{p^4}$, while $\hat{\EC}_{\sf X}$ is defined\footnote{Recall from Section~\ref{sec:superMNT} (or observe in Table~\ref{tab:examples}) that these curves all have $j$-invariants in the ground field, but only find the $p$- or $q$-torsion required for the cycle over the extension field listed.} over $\F_{q^2}$ or $\F_{q^3}$.
\item The curve $E_{\sf X}$, which is such that $\#E_{\sf X}=h \cdot r$, with $r$ a ${\sf Y}$-bit prime. The curve $E_{\sf X}$ is pairing-friendly with respect to $r$; it has embedding degree 4 if it is defined over $\F_p$ and embedding degree 6 if it is defined over $\F_q$. 
\item A bold ${\bf E}$ denotes a non-pairing-friendly curve that is attached to the lollipop. Each example lollipop comes with three non-pairing-friendly incarnations: a cycle of two curves ${\bf E}_{\sf Y}/\F_r$ and $\hat{\bf E}_{\sf Y}/\F_{\hat{r}}$ with $r=\#\hat{\bf E}_{\sf Y}$ and $\hat{r}=\#{\bf E}_{\sf Y}$, a twist-secure short Weierstrass curve  ${\bf E}^{\rm W}_{\sf Y}/\F_r$, and a twist-secure (twisted) Edwards curve ${\bf E}^{\rm Ed}_{\sf Y}/\F_r$.
\end{itemize}

Figure~\ref{fig:types} depicts the five types of lollipops that were observed under our construction. These types depend on the field of definition of the curve $E_{\sf X}$ and the field of definition of the curve $\EC_{\sf X}$. \\

\begin{figure}[ht!]
	\centering
	\subfloat[]{%
		\parbox[c][3cm][b]{0.20\linewidth}{%
		\centering
		\begin{tikzpicture}
	            \node (A) at (0,3.2) {$\EC_{{\sf X}}/\F_{p^2}$};
	            \node (B) at (0,2) {$\hat{\EC}_{{\sf X}}/\F_{q^2}$};
	            \node (C) at (0,1) {${E}_{{\sf X}}/\F_p$};
	            \node (D) at (0,0) {$\mathbf{E}_{{\sf Y}}/\F_r$};
	            \draw[->, bend left=60] (A) to (B);
	            \draw[->, bend left=60] (B) to (A);
	            \draw[->] (C) to (B);
	            \draw[->] (D) to (C);
	        \end{tikzpicture}
	}}\hfil
	\subfloat[]{%
		\parbox[c][3cm][b]{0.20\linewidth}{%
		\centering
		\begin{tikzpicture}
	            \node (A) at (0,3.2) {$\EC_{{\sf X}}/\F_{p^4}$};
	            \node (B) at (0,2) {$\hat{\EC}_{{\sf X}}/\F_{q^2}$};
	            \node (C) at (0,1) {${E}_{{\sf X}}/\F_p$};
	            \node (D) at (0,0) {$\mathbf{E}_{{\sf Y}}/\F_r$};
	            \draw[->, bend left=60] (A) to (B);
	            \draw[->, bend left=60] (B) to (A);
	            \draw[->] (C) to (B);
	            \draw[->] (D) to (C);
	        \end{tikzpicture}
	}}\hfil
	\subfloat[]{%
		\parbox[c][3cm][b]{0.20\linewidth}{%
		\centering
		\begin{tikzpicture}
	            \node (A) at (0,3.2) {$\EC_{{\sf X}}/\F_{p^2}$};
	            \node (B) at (0,2) {$\hat{\EC}_{{\sf X}}/\F_{q^3}$};
	            \node (C) at (0,1) {${E}_{{\sf X}}/\F_p$};
	            \node (D) at (0,0) {$\mathbf{E}_{{\sf Y}}/\F_r$};
	            \draw[->, bend left=60] (A) to (B);
	            \draw[->, bend left=60] (B) to (A);
	            \draw[->] (C) to (B);
	            \draw[->] (D) to (C);
	        \end{tikzpicture}
	}}\hfil
	\subfloat[]{%
		\parbox[c][3cm][b]{0.20\linewidth}{%
		\centering
		\begin{tikzpicture}
	            \node (A) at (0,3.2) {$\hat{\EC}_{{\sf X}}/\F_{q^3}$};
	            \node (B) at (0,2) {$\EC_{{\sf X}}/\F_{p^2}$};
	            \node (C) at (0,1) {${E}_{{\sf X}}/\F_q$};
	            \node (D) at (0,0) {$\mathbf{E}_{{\sf Y}}/\F_r$};
	            \draw[->, bend left=60] (A) to (B);
	            \draw[->, bend left=60] (B) to (A);
	            \draw[->] (C) to (B);
	            \draw[->] (D) to (C);
	        \end{tikzpicture}
	}}\hfil
	\subfloat[]{%
		\parbox[c][3cm][b]{0.20\linewidth}{%
		\centering
		\begin{tikzpicture}
	            \node (A) at (0,3.2) {$\hat{\EC}_{{\sf X}}/\F_{q^3}$};
	            \node (B) at (0,2) {$\EC_{{\sf X}}/\F_{p^4}$};
	            \node (C) at (0,1) {${E}_{{\sf X}}/\F_q$};
	            \node (D) at (0,0) {$\mathbf{E}_{{\sf Y}}/\F_r$};
	            \draw[->, bend left=60] (A) to (B);
	            \draw[->, bend left=60] (B) to (A);
	            \draw[->] (C) to (B);
	            \draw[->] (D) to (C);
	        \end{tikzpicture}
	
	}}
	\caption{The five types of \textsf{lollipop-X-Y} instances, together with the non-pairing-friendly curve $\mathbf{E}_{{\sf Y}}/\F_r$ attached. Further explanation in text. 
}
	\label{fig:types}
\end{figure}
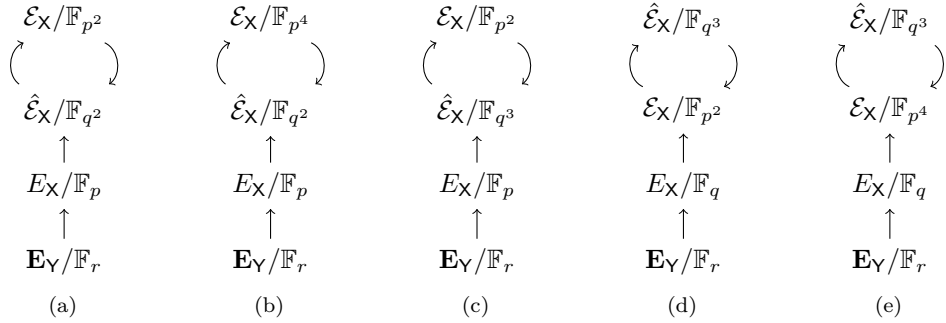

Table~\ref{tab:examples} summarises the 18 lollipops we found. We do not recommend that readers trudge through the details of all 18 of the example lollipops in Appendix~\ref{sec:appendix}, but instead suggest they pick one or two examples of interest after viewing Table~\ref{tab:examples}, or else read the subsection below on the differences between ${\sf MNT298}$ and $\textsf{lollipop-305-158}$. We wanted to include all of the lollipops for which there was an appreciable difference between ${\sf X}$ and ${\sf Y}$, and for which the ECDLP and DLP securities were in the same ballpark. The only exception here is \textsf{lollipop-956-451}, whose ECDLP security is far greater than its DLP security; we included this because it is the only example of practical interest that we found beyond the 128-bit security level.

\begin{table}[ht!] 
\centering
\small
\renewcommand{\tabcolsep}{0.1cm}
\renewcommand{\arraystretch}{1.2}
	\begin{tabular}{|c|c|c|c|c|c|c|c|c|c|c|}
\hline
		Example $\#$ 	&	Fig.			&	 	\multicolumn{4}{|c|}{the (ordinary) stick}				&	\multicolumn{5}{|c|}{the (supersingular) cycle}	\\
\cline{3-11}
(\textsf{lollipop-})			&	  \ref{fig:types}				& $\mathbf{E}_{\sf Y}$ cycle	&  \multicolumn{2}{|c|}{friendly $E_{\sf X}$}		& {\bf sec.}	& \multirow{ 2}{*}{$(u,v)$}	&	\multicolumn{1}{|c|}{${\EC}_{\sf X}/\F_{p^u}$}	&  $\hat{\EC}_{\sf X}/\F_{q^v}$	& \multicolumn{2}{|c|}{{\bf sec.} }	\\
\cline{3-5}\cline{8-11}
\textsf{X-Y}				&	type				& $\hat{D}$	 	& $D$	&    $k$ 	& $E_{\sf X}[r]$	& 	 &  $j(\EC_{\sf X})$ 	& $j(\hat{\EC}_{\sf X})$	&  $\F_{p^4}^\times$ & $\F_{q^6}^\times$ \\
\hline
%%%%%%%%%%%%%%%%%%%%%%%%%%%%%%%%%%%%%%%%%%%%%%%%%%%%%%%%%%%%%%%%%%%%%%%%
\ref{example:305-158}. \textsf{305-158}*	&	(a)	&	3		&54105234	&	4	&	{\bf 78}	&	(2,2)	& 1728	& 0		& {\bf 77}	& 	88	\\
%%%%%%%%%%%%%%%%%%%%%%%%%%%%%%%%%%%%%%%%%%%%%%%%%%%%%%%%%%%%%%%%%%%%%%%%
\ref{example:312-164}. \textsf{312-164} 	&	(b)	&192547		& 6110889	&	4	&	{\bf 81}	&	 (4,2)	& -884736	& 0		& {\bf 78}	&	89	\\
%%%%%%%%%%%%%%%%%%%%%%%%%%%%%%%%%%%%%%%%%%%%%%%%%%%%%%%%%%%%%%%%%%%%%%%%
\ref{example:314-154}. \textsf{314-154}	&	(d)	&	67		&906982043&	6	&	{\bf 76}	&	(2,3)	&	1728	& 8000	& {\bf 79}	&	89	\\
%%%%%%%%%%%%%%%%%%%%%%%%%%%%%%%%%%%%%%%%%%%%%%%%%%%%%%%%%%%%%%%%%%%%%%%%
\ref{example:347-192}. \textsf{347-192}	&	(b)	& 159307	&13173841 &	4	&	{\bf 95}		&	(4,2)	&	-3375	&0 			& {\bf 82}	&	95	\\
%%%%%%%%%%%%%%%%%%%%%%%%%%%%%%%%%%%%%%%%%%%%%%%%%%%%%%%%%%%%%%%%%%%%%%%%
\ref{example:348-168}. \textsf{348-168}	&	(b)	&	43		&8310359121 &	4	&	{\bf 83}	&	(4,2)	&	-3375	& 0 		& {\bf 83}	&	95	\\
%%%%%%%%%%%%%%%%%%%%%%%%%%%%%%%%%%%%%%%%%%%%%%%%%%%%%%%%%%%%%%%%%%%%%%%%
\ref{example:351-196}. \textsf{351-196}*	&	(a)	&	3		&180658	&	4	&	{\bf 97}	&	(2,2)	& 	1728	&  0		& {\bf 83}	&	95	\\
%%%%%%%%%%%%%%%%%%%%%%%%%%%%%%%%%%%%%%%%%%%%%%%%%%%%%%%%%%%%%%%%%%%%%%%%
\ref{example:354-182}. \textsf{354-182}	&	(b)	&	18403	&11984649  &	4	 &	{\bf 90}	&	(4,2)	&	21 \dots 33	&	0 & {\bf 83}	&	95	\\
%%%%%%%%%%%%%%%%%%%%%%%%%%%%%%%%%%%%%%%%%%%%%%%%%%%%%%%%%%%%%%%%%%%%%%%%
\ref{example:360-262}. \textsf{360-262}	&	(c)	&	101971	&6515276374  &	4	&	{\bf 130}	&	(2,3)	&	1728	& 8000						  &	{\bf 84} &	95	\\
%%%%%%%%%%%%%%%%%%%%%%%%%%%%%%%%%%%%%%%%%%%%%%%%%%%%%%%%%%%%%%%%%%%%%%%%
\ref{example:442-201}. \textsf{442-201}*	&	(a)	&	427		&1121454146 &	4	&	{\bf 100}	&	(2,2)	&	1728	&  0		&	{\bf 94} &	106	\\
%%%%%%%%%%%%%%%%%%%%%%%%%%%%%%%%%%%%%%%%%%%%%%%%%%%%%%%%%%%%%%%%%%%%%%%%
\ref{example:447-234}. \textsf{447-234}	&	(d)	&	6339	&	21781087203 &	6	&	{\bf 116}	&	(2,3)	&	1728	&	8000 &	{\bf 95} &	106	\\
%%%%%%%%%%%%%%%%%%%%%%%%%%%%%%%%%%%%%%%%%%%%%%%%%%%%%%%%%%%%%%%%%%%%%%%%
\ref{example:454-179}. \textsf{454-179}	&	(d)	&	355	&	7643719763 &	6	&	{\bf 89}	&	(2,3)	&	1728	&	8000 &	{\bf 96}&	106	\\
%%%%%%%%%%%%%%%%%%%%%%%%%%%%%%%%%%%%%%%%%%%%%%%%%%%%%%%%%%%%%%%%%%%%%%%%
\ref{example:470-217}. \textsf{470-217}	&	(b)	&	2003	&	6965939657 &	4	&	{\bf 108 }	&	(4,2)	&	-32768	& 0	&	{\bf 97} &	109	\\
%%%%%%%%%%%%%%%%%%%%%%%%%%%%%%%%%%%%%%%%%%%%%%%%%%%%%%%%%%%%%%%%%%%%%%%%
\ref{example:489-201}. \textsf{489-201}	&	(b)	&	547	&	372894729 &	4	&	{\bf 100}	&	(4,2)	&	-3375		&0 	& {\bf 98}	&	111	\\
%%%%%%%%%%%%%%%%%%%%%%%%%%%%%%%%%%%%%%%%%%%%%%%%%%%%%%%%%%%%%%%%%%%%%%%%
\ref{example:493-189}. \textsf{493-189}	&	(b)	&	57891&	9926408913&	4	&	{\bf 93}	&	(4,2)	&	-3375		&0 	& {\bf 99}	&	111	\\
%%%%%%%%%%%%%%%%%%%%%%%%%%%%%%%%%%%%%%%%%%%%%%%%%%%%%%%%%%%%%%%%%%%%%%%%
\ref{example:538-235}. \textsf{538-235}	&	(c)	&	22339&	137671666&	4	&	{\bf 117}	&	(2,3)	&	1728		& 8000	&	{\bf 104} &	115	\\
%%%%%%%%%%%%%%%%%%%%%%%%%%%%%%%%%%%%%%%%%%%%%%%%%%%%%%%%%%%%%%%%%%%%%%%%
\ref{example:574-261}. \textsf{574-261}*	&	(a)	&	3019	&	4381481154&	4	&	{\bf 129}	&	(2,2)	&	1728		& 0	&	{\bf 108} &	119	\\
%%%%%%%%%%%%%%%%%%%%%%%%%%%%%%%%%%%%%%%%%%%%%%%%%%%%%%%%%%%%%%%%%%%%%%%%
\ref{example:585-216}. \textsf{585-216}	&	(e)	&	3315	&	975588203 &	6	&	{\bf 107}	&	(4,3)	&	-3375	& $11\dots 85$	&	{\bf 109} &	121	\\
%%%%%%%%%%%%%%%%%%%%%%%%%%%%%%%%%%%%%%%%%%%%%%%%%%%%%%%%%%%%%%%%%%%%%%%%
\ref{example:956-451}. \textsf{956-451}*	&	(a)	&	56731&	40201986 	&	4	&	{\bf 225}	&	(2,2)	&	1728	&	0			&	{\bf 142} &	153	\\
\hline
	\end{tabular}\vspace{0.3cm}
\caption{A summary of the 18 lollipops with details in Appendix~\ref{sec:appendix}.}
	\label{tab:examples}
\end{table}

All of the lollipops have slightly different properties, the most practically relevant of which are given in Table~\ref{tab:examples}. The second column indicates the lollipop type, in reference to the five possibilities in Figure~\ref{fig:types}. The next column gives $\hat{D}$, the discriminant of the non-pairing-friendly cycle ${\bf E}_{\sf Y}/\F_r$ and $\hat{\bf E}_{\sf Y}/\F_{\hat{r}}$; these parameters were computed by tweaking the routine\footnote{See \url{https://github.com/asanso/Bandersnatch/blob/main/python-ref-impl/small-disc-curves.py}.} used to find Masson, Sanso and Zhang's Bandersnatch curve~\cite{bandersnatch}. Observe that Examples~\ref{example:305-158} and~\ref{example:351-196} have $\hat{D}=3$, which means that both ${\bf E}$ and $\hat{\bf E}$ can exploit efficient endomorphisms of the form $\phi \colon (x,y) \rightarrow (\xi x,y)$, similar to the secp/secq~\cite{secpsecq} and Pasta~\cite{pasta} cycles. The only other examples which may have an endomorphism that is efficient enough to use in practice are Example~\ref{example:314-154}, which has $\hat{D}=67$, and Example~\ref{example:348-168}, which has $\hat{D}=43$; these correspond to endomorphisms of degree 17 and 11, respectively.\footnote{If $\{1, \beta\}$ is an integral basis for the ring of integers in $\mathbb{Q}(\sqrt{\hat{D}})$, then the degree of the endomorphism $\phi$ is $N(\beta)-1$; for all of the $\hat{D}$ in Table~\ref{tab:examples}, we get ${\rm deg}(\phi)=(\hat{D}+1)/4$. Stark's algorithm~\cite{stark1973class} can be used to derive the explicit formulas to compute $\phi$.} The next three columns are associated with the pairing-friendly curve, $E_{\sf X}$, in the stick of the lollipop. The first gives the discriminant $D$ of the generalised Pell equation in~\eqref{eq:Pell} that was solved to find the lollipop, and the second gives the embedding degree, $k$, of $E_{\sf X}$ (with respect to $r$). If $k=4$, then $E_{\sf X}$ is defined over $\F_p$ and the CM equation is~\eqref{CMp}; if $k=6$, then $E_{\sf X}$ is defined over $\F_q$ and the CM equation is~\eqref{CMq}. The next column gives the bit-security of $E_{\sf X}[r]$ against Pollard's rho algorithm~\cite{pollard1978monte}, calculated as $\lfloor {\rm log}_2( \sqrt{\pi r/4})\rfloor$. The remaining columns summarise the curves in the supersingular cycle, starting with the fields of definition of $\EC_{\sf X}/\F_{p^u}$ and $\hat{\EC}_{\sf X}/\F_{q^v}$, where $u \in \{2,4\}$ and $v \in \{2,3\}$, depending on the value of $x \bmod{12}$ (see Theorem~\ref{thm:main}). We then give the $j$-invariants of $\EC_{\sf X} \in \F_p$ and $\hat{\EC}_{\sf X} \in \F_q$, which are computed during Br{\"o}ker's algorithm (see Section~\ref{sec:superMNT}). In the last two columns we estimate the DLP security of the lollipop against the special tower number field sieve (S-TNFS), which were obtained using the \texttt{SageMath}~\cite{sagemath} program for TNFS simulation by Guillevic and Singh~\cite{guillevic2021alpha}. All the simulations were run for $10^5$ samples with $\deg h = 2$ because the prime $p$ is given by a polynomial of degree 2 for our curves~\cite[Appendix B]{DBLP:conf/eurocrypt/HousniG22}. Finally, those examples marked with an $*$ are intended to highlight that they correspond to the optimal scenario where $u=v=2$. 

\paragraph{\textsf{MNT298} vs. \textsf{lollipop-305-158}.} Lollipops of pairing-friendly curves allow unbounded recursive pairing-based proof composition and simultaneously allow the initial SNARK circuit arithmetic to be performed over a field whose size is either optimal with respect to a given security level, or much closer to it. We illustrate the idea below by comparing the \textsf{MNT298} cycle from the original work proposing the use of pairing-friendly cycles~\cite{scalable} to Example~\ref{example:305-158} of our construction: \textsf{lollipop-305-158}. 

The \textsf{MNT298} instance targets the 80-bit security level by working with a cycle of ordinary MNT curves defined over the 298-bit primes $p$ and $q$. The corresponding DLP instances lie in $\F_{p^4}^\times$ and $\F_{q^6}^\times$, which lie in the fields of size $1192$ and $1788$ bits, respectively. The main reason this \textsf{MNT298} cycle is suboptimal at the 80-bit security level is that the primes $p$ and $q$ are almost twice as large as the sizes of primes that non-pairing-friendly curves could be defined over at this security level. 

As an alternative, consider~\textsf{lollipop-305-158}, summarised in Table~\ref{tab:examples}. It contains a $2$-cycle of supersingular curves, ${\EC}_{305}/\F_{p^2}$ and $\hat{\EC}_{305}/\F_{q^2}$, where $p$ and $q$ are two 305-bit primes. Just like the MNT construction above, pairings map ECDLP instances to instances of the DLP in $\F_{p^4}^\times$ and $\F_{q^6}^\times$, which are fields of size $1220$ and $1830$ bits, respectively. In this case, however, there is a third curve, $E_{305}/\F_p$, which is ordinary; it is pairing-friendly with respect to a 158-bit prime $r$, and the order-$r$ Weil pairing on $E_{305}$ maps into the same multiplicative group (i.e. $\F_{p^4}^\times$) as the order-$q$ Weil pairing on ${\EC}_{305}$. Moreover, the complexity of the best attack against the ECDLP in $E_{305}(\F_p)[r]$ closely matches the $80$ bits of security offered by the corresponding DLP's in $\F_{p^4}^\times$. On average, Pollard's rho~\cite{pollard1978monte} algorithm solves the ECDLP in $E_{305}(\F_p)[r]$ in $\sqrt{\frac{\pi \cdot r}{4}} \approx 2^{79}$ elliptic curve group operations. 

The crucial point is that $E_{305}/\F_p$ being pairing-friendly with respect to a 158-bit prime $r$ allows us to start the (first layer of the) recursive proof system over $\F_r$, rather than $\F_p$ or $\F_q$. Such a proof system could either work over the field $\F_r$ directly, or else could define non-pairing-friendly curves there, for which there are three scenarios of practical interest; these are depicted in Figure~\ref{fig:compare}. The first option (see Figure~\ref{pic:b}) is to define the short Weierstrass curve $\mathbf{E}_{158}^{\rm W}$ of the form $y^2=x^3-3x+b$, where $b=7032$ is minimal such that $\mathbf{E}_{158}^{\rm W}$ and its quadratic twist have prime order. The second option (see Figure~\ref{pic:c}) is to define the twisted Edwards curve $\mathbf{E}_{158}^{\rm Ed}$ of the form $-x^2+y^2=1+dx^2y^2$, where $d=7821$ is minimal such that the cofactors of the $\mathbf{E}_{158}^{\rm Ed}$ and its quadratic twist are optimally small, i.e. $8$ and $4$, respectively; this construction is analogous to Bowe and Hopwood's Jubjub curve\footnote{See~\url{https://github.com/zkcrypto/jubjub}} defined over the BLS-381 scalar field. The third option (see Figure~\ref{pic:d}) is to define a non-pairing-friendly cycle of prime order curves $\mathbf{E}_{158}$ and $\hat{\mathbf{E}}_{158}$. They both have complex multiplication (CM) discriminant $D=-3$, and can be written as 
$\mathbf{E}_{158}/\F_{r} \colon y^2=x^3+2$ and $\hat{\mathbf{E}}_{158}/\F_{\hat{r}} \colon y^2=x^3+2$, where $\hat{r}$ is also a 158-bit prime such that $\#\mathbf{E}_{158}(\F_r) = \hat{r}$ and $\#\hat{\mathbf{E}}_{158}(\F_{\hat{r}})=r$. Both of these curves come equipped with an efficient endomorphism that can accelerate (multi-)scalar multiplications, analogous to the secp/secq~\cite{secpsecq} and Pasta~\cite{pasta} cycles. Non-pairing-based SNARKs like Bulletproofs~\cite{bulletproofs} could be instantiated using the $\mathbf{E}_{158}$ and $\hat{\mathbf{E}}_{158}$ cycle, similar to their instantiations on the secp/secq and Pasta cycles, but the key difference is that the lollipop of pairing-friendly curves connected to $\mathbf{E}_{158}$ and $\hat{\mathbf{E}}_{158}$ allows the succinct {\sf Groth16}-style proofs to combine, compose and check the proofs on the non-pairing-friendly cycle. 

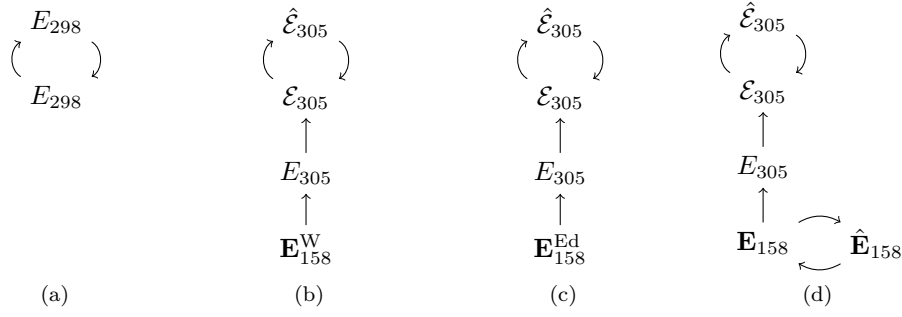
\begin{figure}[!ht]
	\centering
	\subfloat[]{%
		\parbox[c][3cm][b]{0.24\linewidth}{%
		\centering
		\begin{tikzpicture}
	            \node (A) at (0,3.2) {$E_{298}$};
	            \node (B) at (0,2.2) {$E_{298}$};
	            \node (C) at (0,1) {\enspace};
	            \node (D) at (0,0) {\enspace};
	            \draw[->, bend left=60] (A) to (B);
	            \draw[->, bend left=60] (B) to (A);
	        \end{tikzpicture}
	}}\hfil
	\subfloat[]{%
		\parbox[c][3cm][b]{0.24\linewidth}{%
		\centering
		\begin{tikzpicture}
	            \node (A) at (0,3) {$\hat{\EC}_{305}$};
	            \node (B) at (0,2) {$\EC_{305}$};
	            \node (C) at (0,1) {${E}_{305}$};
	            \node (D) at (0,0) {$\mathbf{E}_{158}^{\rm W}$};
	            \draw[->, bend left=60] (A) to (B);
	            \draw[->, bend left=60] (B) to (A);
	            \draw[->] (C) to (B);
	            \draw[->] (D) to (C);
	        \end{tikzpicture}
	\label{pic:b}
	}}\hfil
	\subfloat[]{%
		\parbox[c][3cm][b]{0.24\linewidth}{%
		\centering
		\begin{tikzpicture}
	            \node (A) at (0,3) {$\hat{\EC}_{305}$};
	            \node (B) at (0,2) {$\EC_{305}$};
	            \node (C) at (0,1) {${E}_{305}$};
	            \node (D) at (0,0) {$\mathbf{E}_{158}^{\rm Ed}$};
	            \draw[->, bend left=60] (A) to (B);
	            \draw[->, bend left=60] (B) to (A);
	            \draw[->] (C) to (B);
	            \draw[->] (D) to (C);
	        \end{tikzpicture}
\label{pic:c}
	}}\hfil
	\subfloat[]{%
		\parbox[c][3cm][b]{0.24\linewidth}{%
		\centering
		\begin{tikzpicture}
	            \node (A) at (0,3) {$\hat{\EC}_{305}$};
	            \node (B) at (0,2) {$\EC_{305}$};
	            \node (C) at (0,1) {${E}_{305}$};
	            \node (D) at (0,0) {$\mathbf{E}_{158}$};
		   \node (E) at (1.5,0) {$\hat{\mathbf{E}}_{158}$};
	            \draw[->, bend left=60] (A) to (B);
	            \draw[->, bend left=60] (B) to (A);
	            \draw[->] (C) to (B);
	            \draw[->] (D) to (C);
	            \draw[->, bend left=30] (D) to (E);
	            \draw[->, bend left=30] (E) to (D);
	        \end{tikzpicture}
\label{pic:d}
	}}
	\caption{\textsf{MNT298} vs. (the three incarnations of) \textsf{lollipop-305-158}. Further explanation in text.}
	\label{fig:compare}
\end{figure}

\section{Conclusion}\label{sec: conclusion}

We introduced a novel construction of cycles of pairing-friendly elliptic curves geared towards unbounded recursive pairing-based proof systems. The construction invokes the use of supersingular elliptic curves, rather than ordinary MNT curves. The main caveat of our construction is that the supersingular curves are defined over extension fields. The benefit is that constructing the curves themselves no longer requires the invocation of the CM method, so is not restricted to a finite number of low discriminant instances. We therefore realise an infinite family of pairing-friendly cycles that is more flexible than the MNT construction. 

This flexibility brings up multiple possibilities of interest in the context of proof systems. One such possibility is the construction of lollipops of pairing-friendly curves, which we explored in detail in Sections~\ref{sec:constructing} and~\ref{sec:examples}. 

Future work includes a high-performance library based on the optimal instantiation of Theorem~\ref{thm:main}, as well as addressing some of the lollipop-specific instances we discussed. This includes constructing lollipops that take the smallest field (i.e. $\F_r$) as input, so unbounded recursion can be layered on top of pre-existing elliptic curves or proof infrastructures, as well as more optimised searches that target high-security instances of lollipops.

\iffullversion
    \subsubsection{Acknowledgements.} Part of this work was done while Gaurish was an intern at Microsoft Research. Thanks to Patrick Longa, Andrew Sutherland, and Allan Steel for their technical assistance. Thanks to Michael Naehrig and Greg Zaverucha for several discussions during the preparation of this work. 
\fi

\bibliography{bib}
\appendix

\section{Details of the examples}\label{sec:appendix}

\begin{example}(\textsf{lollipop-305-158}*).\label{example:305-158} 
Solving~\eqref{CMp} with $D=54105234$ finds $x \equiv 10 \bmod{12}$ such that $p$ and $q$ are prime $305$-bit primes and gives a type-(a) lollipop in Fig.~\ref{fig:types}. The curves in the cycle are $\hat{\EC}_{305} \colon y^2=x^3+(\eta+2)$ where  $\F_{q^2}=\F_q(\eta)$ with $\eta^2+2=0$, and ${\EC}_{305} \colon y^2=x^3+(\mu+1)x$ where $\F_{p^2}=\F_p(\mu)$ with $\mu^2+1=0$. The curve in the stick is $E_{305}/\F_p$ has embedding degree 4 with respect to a 158-bit prime $r$. The non-pairing-friendly curves over $\F_r$ are $\mathbf{E}_{158}^{\rm W} \colon y^2=x^3-3x+b$ with $b=7032$ and $\mathbf{E}_{158}^{\rm Ed} \colon -x^2+y^2=1+dx^2y^2$ with $d=7821$. Both curves in the cycle $\mathbf{E}_{158}$ and $\hat{\mathbf{E}}_{158}$ have $\hat{D}=3$ and are equipped with endomorphisms of the form $\phi \colon (x,y) \mapsto (\xi x,y)$. 
\end{example}

\begin{example}(\textsf{lollipop-312-164}).\label{example:312-164} 
Solving~\eqref{CMp} with $D=6110889$ finds $x \equiv 4 \bmod{12}$ such that $p$ and $q$ are prime $312$-bit primes and gives a type-(b) lollipop in Fig.~\ref{fig:types}. Br{\"o}ker's algorithm terminates with $\D=19$ and $H_\D(X)= X + 884736$ outputs a curve ${\EC}_{312}/\F_{p^4}$ with $j({\EC}_{312})=-884736$, where $\F_{p^2}=\F_p(\mu)$ with $\mu^2+13=0$ and $\F_{p^4}=\F_{p^2}(\nu)$ with $\nu^2=\mu$. The other curve in the cycle is $\hat{\EC}_{312}/\F_q \colon y^2=x^3+(\eta+1)$ where  $\F_{q^2}=\F_q(\eta)$ with $\eta^2+3=0$. The curve in the stick is $E_{312}/\F_p$ has embedding degree 4 with respect to a 164-bit prime $r$. The non-pairing-friendly curves over $\F_r$ are $\mathbf{E}_{164}^{\rm W} \colon y^2=x^3-3x+b$ with $b=6457$ and $\mathbf{E}_{164}^{\rm Ed} \colon -x^2+y^2=1+dx^2y^2$ with $d=2709$. Both curves in the cycle $\mathbf{E}_{164}$ and $\hat{\mathbf{E}}_{164}$ have $\hat{D}=192547$. 
\end{example}

\begin{example}(\textsf{lollipop-314-154}).\label{example:314-154}  
Solving~\eqref{CMq} with $D=906982043$ finds $x \equiv 6 \bmod{12}$ such that $p$ and $q$ are prime $314$-bit primes and gives a type-(c) lollipop in Fig.~\ref{fig:types}. Over $\F_q$, Br{\"o}ker's algorithm terminates with $\hat{\D}=-8$ and $H_{\hat{\D}}(X)= X -8000$ and outputs $\hat{\EC}_{314}$ with $j(\hat{\EC}_{314}) =8000$. The other curve in the cycle is ${\EC}_{314} \colon y^2=x^3+(1+\mu)x$  where $\F_{p^2}=\F_p(\mu)$ with $\mu^2+1=0$.  The curve in the stick is $E_{314}/\F_q$ has embedding degree 6 with respect to a 154-bit prime $r$. The non-pairing-friendly curves over $\F_r$ are $\mathbf{E}_{154}^{\rm W} \colon y^2=x^3-3x+b$ with $b=17221$ and $\mathbf{E}_{154}^{\rm Ed} \colon -x^2+y^2=1+dx^2y^2$ with $d=4468$. Both curves in the cycle $\mathbf{E}_{154}$ and $\hat{\mathbf{E}}_{154}$ have $\hat{D}=67$. 
\end{example}

\begin{example}(\textsf{lollipop-347-192}).\label{example:347-192} 
Solving~\eqref{CMp} with $D=13173841$ finds $x \equiv 4 \bmod{12}$ such that $p$ and $q$ are prime $347$-bit primes and gives a type-(b) lollipop in Fig.~\ref{fig:types}. Br{\"o}ker's algorithm terminates with $\D=7$ and $H_\D(X)= X + 3375$ outputs a curve ${\EC}_{347}/\F_{p^4}$ with $j({\EC}_{347})=-3375$, where $\F_{p^2}=\F_p(\mu)$ with $\mu^2+2=0$ and $\F_{p^4}=\F_{p^2}(\nu)$ with $\nu^2=\mu$. The other curve in the cycle is $\hat{\EC}_{347}/\F_q \colon y^2=x^3+(\eta+4)$ where  $\F_{q^2}=\F_q(\eta)$ with $\eta^2+3=0$. The curve in the stick is $E_{347}/\F_p$ has embedding degree 4 with respect to a 192-bit prime $r$. The non-pairing-friendly curves over $\F_r$ are $\mathbf{E}_{192}^{\rm W} \colon y^2=x^3-3x+b$ with $b=11566$ and $\mathbf{E}_{192}^{\rm Ed} \colon -x^2+y^2=1+dx^2y^2$ with $d=69127$. Both curves in the cycle $\mathbf{E}_{192}$ and $\hat{\mathbf{E}}_{192}$ have $\hat{D}=159307$. 
\end{example}

\begin{example}(\textsf{lollipop-348-168}).\label{example:348-168} 
Solving~\eqref{CMp} with $D=8310359121$ finds $x \equiv 4 \bmod{12}$ such that $p$ and $q$ are prime $347$-bit primes and gives a type-(b) lollipop in Fig.~\ref{fig:types}. Br{\"o}ker's algorithm terminates with $\D=7$ and $H_\D(X)= X + 3375$ outputs a curve ${\EC}_{348}/\F_{p^4}$ with $j({\EC}_{348})=-3375$, where $\F_{p^2}=\F_p(\mu)$ with $\mu^2+7=0$ and $\F_{p^4}=\F_{p^2}(\nu)$ with $\nu^2=\mu$. The other curve in the cycle is $\hat{\EC}_{347} \colon y^2=x^3+(\eta+6)$ where  $\F_{q^2}=\F_q(\eta)$ with $\eta^2+3=0$. The curve in the stick is $E_{348}/\F_p$ has embedding degree 4 with respect to a 168-bit prime $r$. The non-pairing-friendly curves over $\F_r$ are $\mathbf{E}_{168}^{\rm W} \colon y^2=x^3-3x+b$ with $b=26688$ and $\mathbf{E}_{168}^{\rm Ed} \colon -x^2+y^2=1+dx^2y^2$ with $d=78971$. Both curves in the cycle $\mathbf{E}_{168}$ and $\hat{\mathbf{E}}_{168}$ have $\hat{D}=43$. 
\end{example}

\begin{example}(\textsf{lollipop-351-196}*).\label{example:351-196} 
Solving~\eqref{CMp} with $D=180658$ finds $x \equiv 10 \bmod{12}$ such that $p$ and $q$ are prime $351$-bit primes and gives a type-(a) lollipop in Fig.~\ref{fig:types}. The curves in the cycle are $\hat{\EC}_{351}/\F_q \colon y^2=x^3+(\eta+4)$ where  $\F_{q^2}=\F_q(\eta)$ with $\eta^2+3=0$, and ${\EC}_{351} \colon y^2=x^3+(\mu+1)x$ where $\F_{p^2}=\F_p(\mu)$ with $\mu^2+1=0$. The curve in the stick is $E_{351}/\F_p$ has embedding degree 4 with respect to a 196-bit prime $r$. The non-pairing-friendly curves over $\F_r$ are $\mathbf{E}_{196}^{\rm W} \colon y^2=x^3-3x+b$ with $b=7193$ and $\mathbf{E}_{196}^{\rm Ed} \colon -x^2+y^2=1+dx^2y^2$ with $d=128421$. 

An unlikely coincidence arose in this example. There were actually two prime order curves over $\F_r$ with CM discriminant $-3$. The curve $\mathbf{E}_{196}/\F_r \colon y^2=x^3-5$ has prime group order $\hat{r}$, forming a cycle with the curve $\hat{\mathbf{E}}_{196}/\F_{\hat{r}} \colon y^2=x^3-5$, while the curve $\mathbf{E}_{{196}'}/\F_r \colon y^2=x^3+28$ has prime group order $\hat{r}'$, forming a cycle with the curve $\hat{\mathbf{E}}_{196}'/\F_{\hat{r}'} \colon y^2=x^3+28$. We depict this in Figure~\ref{fig:double-cycle}, where we dropped the subscripts. All of these curves come equipped with endomorphisms of the form $\phi \colon (x,y) \mapsto (\xi x,y)$. 

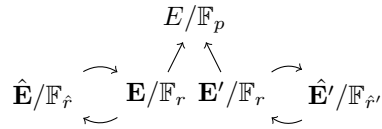
\begin{figure}[ht!]
	\centering
	\begin{tikzpicture}
	            \node (C) at (0,1) {${E}/\F_p$};
	%%%
	            \node (D) at (-0.5,0) {$\mathbf{E}/\F_r$};
		   \node (E) at (-2,0) {$\hat{\mathbf{E}}/\F_{\hat{r}}$};
	            \draw[->] (D) to (C);
	   	\draw[->, bend left=30] (D) to (E);
	     	\draw[->, bend left=30] (E) to (D);
	%%%
	            \node (F) at (0.5,0) {$\mathbf{E}'/\F_{r}$};
		   \node (G) at (2,0) {$\hat{\mathbf{E}'}/\F_{\hat{r}'}$};
	            \draw[->] (F) to (C);
	   	\draw[->, bend left=30] (F) to (G);
	     	\draw[->, bend left=30] (G) to (F);
	\end{tikzpicture}
	\caption{A rare coincidence of two $D=-3$ non-pairing-friendly cycles possible over $\F_r$. }
	\label{fig:double-cycle}
\end{figure}

\end{example}

\begin{example}(\textsf{lollipop-354-182}). \label{example:354-182} 
Solving~\eqref{CMp} with $D=11984649$ finds $x \equiv 4 \bmod{12}$ such that $p$ and $q$ are prime $354$-bit primes and gives a type-(b) lollipop in Fig.~\ref{fig:types}. Br{\"o}ker's algorithm terminates with $\D=31$ and $H_\D(X)$ of degree 3, and outputs a curve ${\EC}_{354}/\F_{p^4}$ with $j({\EC}_{354})=21 \dots 33$, where $\F_{p^2}=\F_p(\mu)$ with $\mu^2+2=0$ and $\F_{p^4}=\F_{p^2}(\nu)$ with $\nu^2=\mu$. The other curve in the cycle is $\hat{\EC}_{354}/\F_q \colon y^2=x^3+(\eta+1)$ where  $\F_{q^2}=\F_q(\eta)$ with $\eta^2+3=0$. The curve in the stick is $E_{354}/\F_p$ has embedding degree 4 with respect to a 182-bit prime $r$. The non-pairing-friendly curves over $\F_r$ are $\mathbf{E}_{182}^{\rm W} \colon y^2=x^3-3x+b$ with $b=7906$ and $\mathbf{E}_{182}^{\rm Ed} \colon -x^2+y^2=1+dx^2y^2$ with $d=2758$. Both curves in the cycle $\mathbf{E}_{182}$ and $\hat{\mathbf{E}}_{182}$ have $\hat{D}=18403$. 
\end{example}

\begin{example}(\textsf{lollipop-360-262}). \label{example:360-262} Solving~\eqref{CMp} with $D=6515276374$ finds $x \equiv 6 \bmod{12}$ such that $p$ and $q$ are prime $360$-bit primes and gives a type-(a) lollipop in Fig.~\ref{fig:types}. Over $\F_q$, Br{\"o}ker's algorithm terminates with $\hat{\D}=-8$ and $H_{\hat{\D}}(X)= X -8000$ and outputs $\hat{\EC}_{360}$ with $j(\hat{\EC}_{360}) =8000$. The other curve in the cycle is ${\EC}_{360} \colon y^2=x^3+(2+\mu)x$  where $\F_{p^2}=\F_p(\mu)$ with $\mu^2+1=0$. The curve in the stick is $E_{360}/\F_p$ has embedding degree 4 with respect to a 262-bit prime $r$. The non-pairing-friendly curves over $\F_r$ are $\mathbf{E}_{262}^{\rm W} \colon y^2=x^3-3x+b$ with $b=43954$ and $\mathbf{E}_{262}^{\rm Ed} \colon -x^2+y^2=1+dx^2y^2$ with $d=119965$. The curves in the cycle $\mathbf{E}_{262}$ and $\hat{\mathbf{E}}_{262}$ have $\hat{D}=101971$. 
\end{example}

\begin{example}(\textsf{lollipop-442-201}*). \label{example:442-201} Solving~\eqref{CMp} with $D=1121454146$ finds $x \equiv 10 \bmod{12}$ such that $p$ and $q$ are prime $442$-bit primes and gives a type-(a) lollipop in Fig.~\ref{fig:types}. The curves in the cycle are $\hat{\EC}_{442}/\F_q \colon y^2=x^3+(\eta+2)$ where  $\F_{q^2}=\F_q(\eta)$ with $\eta^2+2=0$, and ${\EC}_{442} \colon y^2=x^3+(\mu+1)x$ where $\F_{p^2}=\F_p(\mu)$ with $\mu^2+1=0$. The curve in the stick is $E_{442}/\F_p$ has embedding degree 4 with respect to a 201-bit prime $r$. The non-pairing-friendly curves over $\F_r$ are $\mathbf{E}_{201}^{\rm W} \colon y^2=x^3-3x+b$ with $b=1858$ and $\mathbf{E}_{201}^{\rm Ed} \colon -x^2+y^2=1+dx^2y^2$ with $d=80379$. Both curves in the cycle $\mathbf{E}_{201}$ and $\hat{\mathbf{E}}_{201}$ have $\hat{D}=18403$. 
\end{example}

\begin{example}(\textsf{lollipop-447-234}). \label{example:447-234} Solving~\eqref{CMq} with $D=8162838387$ finds $x \equiv 6 \bmod{12}$ such that $p$ and $q$ are prime $447$-bit primes and gives a type-(c) lollipop in Fig.~\ref{fig:types}. Over $\F_q$, Br{\"o}ker's algorithm terminates with $\hat{\D}=-8$ and $H_{\hat{\D}}(X)= X -8000$ and outputs $\hat{\EC}_{447}$ with $j(\hat{\EC}_{447}) =8000$. The other curve in the cycle is ${\EC}_{447} \colon y^2=x^3+(5+\mu)x$  where $\F_{p^2}=\F_p(\mu)$ with $\mu^2+1=0$.  The curve in the stick is $E_{447}/\F_q$ has embedding degree 6 with respect to a 234-bit prime $r$. The non-pairing-friendly curves over $\F_r$ are $\mathbf{E}_{234}^{\rm W} \colon y^2=x^3-3x+b$ with $b=10885$ and $\mathbf{E}_{234}^{\rm Ed} \colon x^2+y^2=1+dx^2y^2$ with $d=-54523$. Both curves in the cycle $\mathbf{E}_{234}$ and $\hat{\mathbf{E}}_{234}$ have $\hat{D}=6339$. 
\end{example}

\begin{example}(\textsf{lollipop-454-179}). \label{example:454-179} Solving~\eqref{CMq} with $D=7643719763$ finds $x \equiv 6 \bmod{12}$ such that $p$ and $q$ are prime $454$-bit primes and gives a type-(c) lollipop in Fig.~\ref{fig:types}. Over $\F_q$, Br{\"o}ker's algorithm terminates with $\hat{\D}=-8$ and $H_{\hat{\D}}(X)= X-8000$ and outputs $\hat{\EC}_{454}$ with $j(\hat{\EC}_{454}) =8000$. The other curve in the cycle is ${\EC}_{454} \colon y^2=x^3+(5+\mu)x$  where $\F_{p^2}=\F_p(\mu)$ with $\mu^2+1=0$.  The curve in the stick is $E_{454}/\F_q$ has embedding degree 6 with respect to a 179-bit prime $r$. The non-pairing-friendly curves over $\F_r$ are $\mathbf{E}_{179}^{\rm W} \colon y^2=x^3-3x+b$ with $b=10827$ and $\mathbf{E}_{179}^{\rm Ed} \colon x^2+y^2=1+dx^2y^2$ with $d=-68661$. Both curves in the cycle $\mathbf{E}_{179}$ and $\hat{\mathbf{E}}_{179}$ have $\hat{D}=355$. 
\end{example}

\begin{example}(\textsf{lollipop-470-217}). \label{example:470-217} Solving~\eqref{CMp} with $D=6965939657$ finds $x \equiv 4 \bmod{12}$ such that $p$ and $q$ are prime $470$-bit primes and gives a type-(b) lollipop in Fig.~\ref{fig:types}. Br{\"o}ker's algorithm terminates with $\D=11$ and $H_\D(X)= X + 32768$ and outputs a curve ${\EC}_{470}/\F_{p^4}$ with $j({\EC}_{470})=-32768$, where $\F_{p^2}=\F_p(\mu)$ with $\mu^2+11=0$ and $\F_{p^4}=\F_{p^2}(\nu)$ with $\nu^2=\mu$. The other curve in the cycle is $\hat{\EC}_{470}/\F_q \colon y^2=x^3+(\eta+2)$ where  $\F_{q^2}=\F_q(\eta)$ with $\eta^2+3=0$. The curve in the stick is $E_{470}/\F_p$ has embedding degree 4 with respect to a 217-bit prime $r$. The non-pairing-friendly curves over $\F_r$ are $\mathbf{E}_{217}^{\rm W} \colon y^2=x^3-3x+b$ with $b=72802$ and $\mathbf{E}_{217}^{\rm Ed} \colon -x^2+y^2=1+dx^2y^2$ with $d=70192$. Both curves in the cycle $\mathbf{E}_{217}$ and $\hat{\mathbf{E}}_{217}$ have $\hat{D}=2003$.
\end{example}

\begin{example}(\textsf{lollipop-489-201}). \label{example:489-201} Solving~\eqref{CMp} with $D=372894729$ finds $x \equiv 4 \bmod{12}$ such that $p$ and $q$ are prime $489$-bit primes and gives a type-(b) lollipop in Fig.~\ref{fig:types}. Br{\"o}ker's algorithm terminates with $\D=7$ and $H_\D(X)=  X + 3375$ and outputs a curve ${\EC}_{489}/\F_{p^4}$ with $j({\EC}_{489})=-3375$, where $\F_{p^2}=\F_p(\mu)$ with $\mu^2+2=0$ and $\F_{p^4}=\F_{p^2}(\nu)$ with $\nu^2=\mu$. The other curve in the cycle is $\hat{\EC}_{489}/\F_q \colon y^2=x^3+(\eta+1)$ where  $\F_{q^2}=\F_q(\eta)$ with $\eta^2+3=0$. The curve in the stick is $E_{489}/\F_p$ has embedding degree 4 with respect to a 201-bit prime $r$. The non-pairing-friendly curves over $\F_r$ are $\mathbf{E}_{201}^{\rm W} \colon y^2=x^3-3x+b$ with $b=4438$ and $\mathbf{E}_{201}^{\rm Ed} \colon -x^2+y^2=1+dx^2y^2$ with $d=96027$. Both curves in the cycle $\mathbf{E}_{201}$ and $\hat{\mathbf{E}}_{201}$ have $\hat{D}=547$. 
\end{example}

\begin{example}(\textsf{lollipop-493-189}). \label{example:493-189} Solving~\eqref{CMp} with $D=9926408913$ finds $x \equiv 4 \bmod{12}$ such that $p$ and $q$ are prime $493$-bit primes and gives a type-(b) lollipop in Fig.~\ref{fig:types}. Br{\"o}ker's algorithm terminates with $\D=7$ and $H_\D(X)=  X + 3375$ and outputs a curve ${\EC}_{493}/\F_{p^4}$ with $j({\EC}_{493})=-3375$, where $\F_{p^2}=\F_p(\mu)$ with $\mu^2+2=0$ and $\F_{p^4}=\F_{p^2}(\nu)$ with $\nu^2=\mu$. The other curve in the cycle is $\hat{\EC}_{493}/\F_q \colon y^2=x^3+(\eta+1)$ where  $\F_{q^2}=\F_q(\eta)$ with $\eta^2+3=0$. The curve in the stick is $E_{493}/\F_p$ has embedding degree 4 with respect to a 189-bit prime $r$. The non-pairing-friendly curves over $\F_r$ are $\mathbf{E}_{189}^{\rm W} \colon y^2=x^3-3x+b$ with $b=40288$ and $\mathbf{E}_{189}^{\rm Ed} \colon -x^2+y^2=1+dx^2y^2$ with $d=54091$. Both curves in the cycle $\mathbf{E}_{189}$ and $\hat{\mathbf{E}}_{189}$ have $\hat{D}=57891$. 
\end{example}

\begin{example}(\textsf{lollipop-538-235}). \label{example:538-235} Solving~\eqref{CMp} with $D=137671666$ finds $x \equiv 6 \bmod{12}$ such that $p$ and $q$ are prime $538$-bit primes and gives a type-(a) lollipop in Fig.~\ref{fig:types}. Over $\F_q$, Br{\"o}ker's algorithm terminates with $\hat{\D}=-8$ and $H_{\hat{\D}}(X)= X -8000$ and outputs $\hat{\EC}_{538}$ with $j(\hat{\EC}_{538}) =8000$. The other curve in the cycle is ${\EC}_{538} \colon y^2=x^3+(1+\mu)x$  where $\F_{p^2}=\F_p(\mu)$ with $\mu^2+1=0$. The curve in the stick is $E_{538}/\F_p$ has embedding degree 4 with respect to a 235-bit prime $r$. The non-pairing-friendly curves over $\F_r$ are $\mathbf{E}_{235}^{\rm W} \colon y^2=x^3-3x+b$ with $b=11095$ and $\mathbf{E}_{235}^{\rm Ed} \colon x^2+y^2=1+dx^2y^2$ with $d=101828$. The curves in the cycle $\mathbf{E}_{235}$ and $\hat{\mathbf{E}}_{235}$ have $\hat{D}=22339$. 
\end{example}

\begin{example}(\textsf{lollipop-574-261}*). \label{example:574-261} Solving~\eqref{CMp} with $D=4381481154$ finds $x \equiv 10 \bmod{12}$ such that $p$ and $q$ are prime $574$-bit primes and gives a type-(a) lollipop in Fig.~\ref{fig:types}. The curves in the cycle are $\hat{\EC}_{574}/\F_q \colon y^2=x^3+(\eta+2)$ where  $\F_{q^2}=\F_q(\eta)$ with $\eta^2+2=0$, and ${\EC}_{574} \colon y^2=x^3+(\mu+1)x$ where $\F_{p^2}=\F_p(\mu)$ with $\mu^2+1=0$. The curve in the stick is $E_{574}/\F_p$ has embedding degree 4 with respect to a 261-bit prime $r$. The non-pairing-friendly curves over $\F_r$ are $\mathbf{E}_{261}^{\rm W} \colon y^2=x^3-3x+b$ with $b=7182$ and $\mathbf{E}_{261}^{\rm Ed} \colon -x^2+y^2=1+dx^2y^2$ with $d=75745$. Both curves in the cycle $\mathbf{E}_{261}$ and $\hat{\mathbf{E}}_{261}$ have $\hat{D}=3019$. 
\end{example}

\begin{example}(\textsf{lollipop-585-216}). \label{example:585-216} Solving~\eqref{CMq} with $D=975588203$ finds $x \equiv 0 \bmod{12}$ such that $p$ and $q$ are prime $585$-bit primes and gives a type-(d) lollipop in Fig.~\ref{fig:types}. Br{\"o}ker's algorithm terminates with $\D=7$ and $H_\D(X)=  X+3375$ and outputs a curve ${\EC}_{585}/\F_{p^4}$ with $j({\EC}_{585})=-3375$, where $\F_{p^2}=\F_p(\mu)$ with $\mu^2+2=0$ and $\F_{p^4}=\F_{p^2}(\nu)$ with $\nu^2=\mu$.  Over $\F_q$, Br{\"o}ker's algorithm terminates with $\hat{\D}=47$ and $H_{\hat{\D}}(X)$ of degree 5, and outputs $\hat{\EC}_{585}$ with $j(\hat{\EC}_{585})=11 \dots 85$. The curve in the stick is $E_{585}/\F_q$ has embedding degree 6 with respect to a 216-bit prime $r$. The non-pairing-friendly curves over $\F_r$ are $\mathbf{E}_{216}^{\rm W} \colon y^2=x^3-3x+b$ with $b=8146$ and $\mathbf{E}_{216}^{\rm Ed} \colon x^2+y^2=1+dx^2y^2$ with $d=-36607$. Both curves in the cycle $\mathbf{E}_{216}$ and $\hat{\mathbf{E}}_{216}$ have $\hat{D}=3315$. 
\end{example}

\begin{example}(\textsf{lollipop-956-451}*). \label{example:956-451} Solving~\eqref{CMp} with $D=40201986$ finds $x \equiv 10 \bmod{12}$ such that $p$ and $q$ are prime $956$-bit primes and gives a type-(a) lollipop in Fig.~\ref{fig:types}. The curves in the cycle are $\hat{\EC}_{956}/\F_q \colon y^2=x^3+(\eta+2)$ where  $\F_{q^2}=\F_q(\eta)$ with $\eta^2+2=0$, and ${\EC}_{956} \colon y^2=x^3+(\mu+1)x$ where $\F_{p^2}=\F_p(\mu)$ with $\mu^2+1=0$. The curve in the stick is $E_{956}/\F_p$ has embedding degree 4 with respect to a 451-bit prime $r$. The non-pairing-friendly curves over $\F_r$ are $\mathbf{E}_{451}^{\rm W} \colon y^2=x^3-3x+b$ with $b=146441$ and $\mathbf{E}_{451}^{\rm Ed} \colon -x^2+y^2=1+dx^2y^2$ with $d=104359$. Both curves in the cycle $\mathbf{E}_{451}$ and $\hat{\mathbf{E}}_{451}$ have $\hat{D}=56731$. Given the discrepancy in the DLP and ECDLP security for this lollipop (see Table~\ref{tab:examples}), it is possible that a composite order curve ${\bf E}/\F_r$ with a prime order subgroup of closer to 256-bits is preferable. In this case it should be possible to find such a curve that has a discriminant low enough to exploit endomorphisms, analogous to the Bandersnatch curve~\cite{bandersnatch} (but with a larger cofactor). 
\end{example}

\end{document}